\documentclass[a4paper, 12pt]{amsart}
\usepackage{amsmath,amssymb,amsthm}
\usepackage{amscd}
\usepackage{amsmath}
\usepackage{amssymb}
\usepackage{amsfonts}
\usepackage{cases}
\usepackage{graphicx}
\usepackage{color}
\usepackage{mathrsfs}
\usepackage{amsthm}
\usepackage{enumerate}
\usepackage[all]{xy} 
\usepackage{setspace}
\usepackage{latexsym}
\usepackage{array,enumerate}

\DeclareMathAlphabet{\mathpzc}{OT1}{pzc}{m}{it}
\newtheorem{defi}{Definition}[section]
\newtheorem{theo}[defi]{Theorem}
\newtheorem{prop}[defi]{Proposition}
\newtheorem{coro}[defi]{Corollary}
\newtheorem{lemm}[defi]{Lemma}
\newtheorem{exam}[defi]{Example}

\newtheorem{rema}[defi]{Remark}

\makeatletter

\@addtoreset{equation}{section}

\title{E$_0$-semigroups and product systems of W$^*$-bimodules}

\author{Yusuke Sawada}
\address{Graduate school of mathematics, Nagoya University, Chikusaku, Nagoya, 464-8602, Japan}
\email{m14017c@math.nagoya-u.ac.jp}
\begin{document}
\begin{abstract}
Product systems have been originally
introduced to classify E$_0$-semigroups on type I factors by Arveson. We develop the classification theory of E$_0$-semigroups on a general von Neumann algebra and the dilation theory of CP$_0$-semigroups in terms of W$^*$-bimodules. For this, we provide a notion of product system of W$^*$-bimodules. This is a W$^*$-bimodule version of Arveson's and Bhat-Skeide's product systems. There exists a one-to-one correspondence between CP$_0$-semigroups and units of product systems of W$^*$-bimodules. The correspondence implies a construction of a dilation of a given CP$_0$-semigroup, a classification of E$_0$-semigroups on a von Neumann algebra up to cocycle equivalence and a relationship between Bhat-Skeide's and Muhly-Solel's constructions of minimal dilations of CP$_0$-semigroups.
\end{abstract}
\maketitle
\section{Introduction}
E$_0$-semigroups naturally arise in the quantum field theory. An E$_0$-semigroup is a semigroup of normal $*$-endomorphisms on a von Neumann algebra with $\sigma$-weak continuity, and the study of E$_0$-semigroups have been initiated by Powers in \cite{powe88}. In \cite{arve89}, Arveson has provided the notion of product system and associated a product system with an E$_0$-semigroup on a type I factor. A product system $\{\mathcal{H}_t\}_{t>0}$ is a measurable family of Hilbert spaces $\mathcal{H}_t$ parameterized by positive real numbers equipped with isomorphisms $\mathcal{H}_s\otimes\mathcal{H}_t\cong\mathcal{H}_{s+t}$ with the associativity. Note that this is not his original definition, however they are the essentially same (see \cite{lieb09}). He also classified E$_0$-semigroups on type I factors by product systems up to cocycle conjugacy. E$_0$-semigroups on type I factors are roughly divided into type I, II and III by units of associated product systems. Every product systems have a numerical index and type I E$_0$-semigroups and product systems are completely classified by their indexes. We refer the reader to his monograph \cite{arve03} for the physical background of E$_0$-semigroups and the theory of product systems. The theory of E$_0$-semigroups on von Neumann algebras which are not type I factors, has often been developed in terms of Hilbert modules. A Bhat-Skeide's product system introduced in \cite{bhat-skei00} is a family $\{E_t\}_{t\geq0}$ of Hilbert bimodules over a C$^*$-algebra satisfying a similar property with Arveson's one with respect to tensor products of Hilbert bimodules. They classified E$_0$-semigroups on a C$^*$-algebra by their product systems up to cocycle equivalence. In \cite{alev04}, Alevras has associated a product system of Hilbert bimodules with each E$_0$-semigroup on a II$_1$ factor by a different way from Bhat-Skeide's one and they form a complete invariant. In Skeide's monograph \cite{skei16}, we have the classification theory of E$_0$-semigroups on the algebra $\mathcal{B}^a(E)$ of all adjointable right $A$-linear maps on a Hilbert (von Neumann) $A$-module $E$. On the other hand, Margetts-Srinivasan have introduced other invariants of E$_0$-semigroups on II$_1$ factors in \cite{marg-srin13}, and they have investigated non-cocycle conjugate E$_0$-semigroups on factors in \cite{marg-srin17} by using the modular conjugation of Tomita-Takesaki theory.

A CP$_0$-semigroups is a $\sigma$-weakly continuous semigroup of normal completely positive maps on a von Neumann algebra. The theory of Arveson's product systems influenced the constructions of minimal dilations of CP$_0$-semigroups. Roughly speaking, a dilation of a CP$_0$-semigroup is an extension of it to an E$_0$-semigroup in a suitable sense. Stinespring's dilation theorem can not be applied to CP$_0$-semigroups, and some researchers have shown an existence of the minimal dilation of a given CP$_0$-semigroup gradually. In \cite{bhat96} and \cite{bhat99}, Bhat has shown it in the cases when $M$ is $\mathcal{B}(\mathcal{H})$ and a C$^*$-algebra, respectively, in which we do not assume the $\sigma$-weakly continuity for semigroups. In \cite{bhat-skei00}, Bhat-Skeide constructed minimal dilations by a method which is valid for both of the von Neumann algebra case and the C$^*$-algebra case. Also, we know Muhly-Solel's (\cite{muhl-sole02}) and Arveson's (\cite{arve03}) constructions, which differ from each other, of the minimal dilation of a CP$_0$-semigroup on a von Neumann algebra. It is more difficult to construct an example of E$_0$-semigroups than CP$_0$-semigroups in general, however the existence of (minimal) dilations gives rise to E$_0$-semigroups from CP$_0$-semigroups. This is one of benefits of the dilation theory. Also, in \cite{sawa17}, we have clarified a direct relationships between Bhat-Skeide's and Muhly-Solel's constructions of the minimal dilation of a discrete CP$_0$-semigroup, which is different from one described by Skeide's commutant duality in \cite{skei03} and \cite{skei06}.

There have been no approaches to the classification theory of E$_0$-semigroups on a von Neumann algebra and the dilation theory of CP$_0$-semigroups by the W$^*$-bimodule (which is not von Neumann bimodule) theory. In this paper, we attempt to give a W$^*$-bimodule approach to their field by a way reflected by Bhat-Skeide's works in \cite{bhat-skei00}. 

We give an outline of this paper. We will recall the notions of W$^*$-bimodule, relative tensor product, CP$_0$-semigroup and E$_0$-semigroup in Section \ref{Preliminaries}.

In Section \ref{Product systems of W$^*$-bimodules}, we will provide a concept of product system of W$^*$-bimodules, where adopted tensor products are relative tensor product introduced by Connes\cite{conn94}. This is a direct extension of Arveson's product system. A unit $\Xi$ of a product system $H$ of W$^*$-$M$-bimodules induces an E$_0$-semigroup on ${\rm
End}(\mathfrak{H}_M)$, where $\mathfrak{H}$ is the inductive limit of $H$ with respect to parameters. The E$_0$-semigroup is called the dilation of the pair $(H,\Xi)$. We prove that cocycles of the dilation of the pair $(H,\Xi)$ and units of $H$ are the essentially same.

In Section \ref{CP$_0$-semigroups and units of product systems of W$^*$-bimodules}, we will find a one-to-one correspondence between CP$_0$-semigroups on a von Neumann algebra $M$ and pairs of product systems of W$^*$-$M$-bimodules and units up to unit preserving isomorphism. The correspondence enables as to translate the $\sigma$-weak continuity of CP$_0$-semigroups into a continuity of units. Also, the dilation of the pair associated with a given CP$_0$-semigroup $T$, gives a dilation of $T$. The product system of W$^*$-bimodules associated with a CP$_0$-semigroup $T$ describes a relation between Bhat-Skeide's and Muhly-Solel's constructions of the minimal dilation of $T$. This is an extension to the continuous case of the relation in the discrete case in \cite{sawa17}. Some relationships among the two constructions and Arveson's construction have not been clarified yet.

In Section \ref{Heat semigroups on manifolds and product systems}, we consider the product system associated the heat semigroup $\{e^{t\Delta}\}_{t\geq0}$ given by the Laplacian $\Delta$ on a compact Riemannian manifold, and its dilation by the method in Section $4$, as an example. We will show that the W$^*$-bimodules appearing in the construction of the product system associated with $\{e^{t\Delta}\}_{t\geq0}$ are realized as $L^2$-spaces with respect to measures given by the heat kernel. We will reconstruct the dilation in more detail under this identification. 

We can get the product system $H^\alpha$ of W$^*$-bimodules (and the unit) from an E$_0$-semigroup $\alpha$ on a von Neumann algebra $M$ as CP$_0$-semigroups by the above correspondence. We will classify E$_0$-semigroups on $M$ by product systems of W$^*$-bimodules: two E$_0$-semigroups $\alpha$ and $\beta$ on $M$ are cocycle equivalent if and only if $H^\alpha\cong
H^\beta$ in Section \ref{Classification of E$_0$-semigroups}. Hence, this enables as to classify E$_0$-semigroups up to cocycle conjugacy by product systems of W$^*$-bimodules. Also, we will get a unit of a given E$_0$-semigroup $\theta$ on II$_1$ factor from a unit of the product system $H^\theta$ associated with $\theta$.

\section{Preliminaries}\label{Preliminaries}
In this section, we recall the notions of W$^*$-bimodule, relative tensor product, CP$_0$-semigroup, E$_0$-semigroup, tensor product related to CP$_0$-semigroup and partition, which will be used in the later sections.

W$^*$-bimodules are Hilbert spaces on which von Neumann algebras act from the left and the right. More precisely, for von Neumann algebras $N$ and $M$, a Hilbert space $\mathcal{H}$ with normal $*$-representations of $N$ and the opposite von Neumann algebra $M^\circ$ of $M$ is a W$^*$-$N$-$M$-bimodule if their representations commute. When $N=\mathbb{C}$ or $M=\mathbb{C}$, we call $\mathcal{H}$ a right W$^*$-$M$-module or a left W$^*$-$N$-module, respectively. We write a W$^*$-$N$-$M$-bimodule, a right W$^*$-$M$-module and a left W$^*$-$N$-module by ${}_N\mathcal{H}_M,\ \mathcal{H}_M$ and ${}_N\mathcal{H}$, respectively.

Let $N$ be a von Neumann algebra, $\mathcal{H}_N$ and $\mathcal{K}_N$ be right W$^*$-$N$-modules, and ${}_N\mathcal{H}'$ and ${}_N\mathcal{K}'$ be left W$^*$-$N$-modules. $\mbox{\rm{Hom}}(\mathcal{H}_N,\mathcal{K}_N)$ and $\mbox{\rm{Hom}}({}_N\mathcal{H}',{}_N\mathcal{K}')$ are the sets of all right and left $N$-linear bounded maps, respectively. If $\mathcal{H}=\mathcal{K}$ and $\mathcal{H}'=\mathcal{K}'$, they are denoted by $\mbox{\rm{End}}(\mathcal{H}_N)$ and $\mbox{\rm{End}}({}_N\mathcal{H}')$, respectively.

We denote the standard space of a von Neumann algebra $M$ by $L^2(M)$. The standard space $L^2(M)$ contains all left and right GNS-spaces and we have $[\phi]\overline{M\psi^\frac{1}{2}}=\overline{\phi^\frac{1}{2}M}[\psi]$ in $L^2(M)$ and $\overline{\phi^\frac{1}{2}M}=[\phi]L^2(M)$ for all $\phi,\psi\in
M_*^+$. In particular, we have $\overline{\phi^\frac{1}{2}M}=L^2(M)=\overline{M\phi^\frac{1}{2}}$ for each faithful $\phi\in
M_*$. This observation will be helpful under the assumption which a von Neumann algebra has a faithful normal state in the later sections. We refer the reader to \cite[Chapter IX]{take03}, \cite{yama14}, \cite{yama92} and \cite{yama94} for details of the definition and properties of standard spaces included in the modular theory.

Now, we shall recall (left) relative tensor products. For more details, see \cite[Chapter 5, Appendix B]{conn94}, \cite{sauv83} or \cite[Chapter IX, Section 3]{take03}. Suppose $\mathcal{H}$ is a W$^*$-$M$-$N$-bimodule and $\mathcal{K}$ is a W$^*$-$N$-$P$-bimodule. Let $\phi$ be a faithful normal state on $N$. A vector $\xi\in\mathcal{H}$ is called a (left) $\phi$-bounded vector if there is $c>0$ such that $\|\xi
x\|\leq
c\|\phi^{\frac{1}{2}}
x\|$ for all $x\in
M$. We denote the set of all $\phi$-bounded vectors in $\mathcal{H}$ by $\mathcal{D}(\mathcal{H};\phi)$. The (left) relative tensor product $\mathcal{H}\otimes_\phi^N\mathcal{K}$ is the completion $\overline{\mathcal{D}(\mathcal{H};\phi)\otimes_{{\rm
alg}}\mathcal{K}}$ with respect to an inner product defined by
\[
\langle\xi_1\phi^{-\frac{1}{2}}\eta_1,\xi_2\phi^{-\frac{1}{2}}\eta_2\rangle=\langle\eta_1,\pi_\phi(\xi_1)^*\pi_\phi(\xi_2)\eta_2\rangle,
\]
for each $\xi_1,\xi_2\in\mathcal{D}(\mathcal{H};\phi)$ and $\eta_1,\eta_2\in\mathcal{K}$, where $\pi_\phi(\xi):L^2(N)\ni\phi^\frac{1}{2}x\to\xi
x\in\mathcal{H}$ and we usually use a notation $\xi\phi^{-\frac{1}{2}}\eta$ rather than $\xi\otimes\eta$. Also, we can define the right relative tensor product by right $\phi$-bounded vectors. 
\begin{rema}
Left and right relative tensor products can be defined by the way which is independent on a choice of $\phi$ {\rm(}see \cite{bail-deni-have88}{\rm)}. If we denote the left and the right relative tensor product by $\mathcal{H}\otimes_l^N\mathcal{K}$ and $\mathcal{H}\otimes_r^N\mathcal{K}$, respectively for W$^*$-bimodules ${}_M\mathcal{H}_N$ and ${}_N\mathcal{K}_P$, we already know the W$^*$-bimodule isomorphism $\mathcal{H}\otimes_l^N\mathcal{K}\cong\mathcal{H}\otimes_r^N\mathcal{K}$. In \cite{sawa-yama17}, we have constructed the isomorphism $\mathcal{H}\otimes_l^N\mathcal{K}\cong\mathcal{H}\otimes_r^N\mathcal{K}$ by the canonical way, and shown that the two W$^*$-bicategories of W$^*$-bimmodules with left and right tensor products as tensor functors are monoidally equivalent.
\end{rema}

Now, we provide the basic notions related with CP$_0$-semigroups and E$_0$-semigroups. A family $T=\{T_t\}_{t\geq0}$ of normal UCP-maps $T_t$ on a von Neumann algebra $M$ is called a CP$_0$-semigroup if $T_0={\rm
id}_M,\ T_sT_t=T_{s+t}$ for all $s,t\geq0$, and for every $x\in
M$ and $\phi\in
M_*$, the function $\phi(T_t(x))$ on $[0,\infty)$ is continuous. If each $T_t$ is a $*$-homomorphism, $T$ is called an E$_0$-semigroup. A CP$_0$-semigroup (E$_0$-semigroup) without the continuity is called an algebraic CP$_0$-semigroup (algebraic E$_0$-semigroup, respectively).
\begin{exam}\label{cpexam0}
Let $\{v_t\}_{t\geq0}$ be a family of isometries $v_t$ in a von Neumann algebra $M$ such that $v_{s+t}=v_sv_t$ for all $s,t\geq0$ and $v_0=1_M$. Suppose $\{v_t\}_{t\geq0}$ is strongly continuous with respect to the parameter. If we define $T=\{T_t\}_{t\geq0}$ by $T_t(x)=v_t^*xv_t$ for each $x\in
M$ and $t\geq0$, then $T$ is a CP$_0$-semigroup. If each $v_t$ is unitary, $T$ is an E$_0$-semigroup. 
\end{exam}
\begin{exam}\label{heat}
The CCR heat flow is a CP$_0$-semigroup $T$ which has the noncommutative Laplacian $\Delta$ as generators. This will be immediately and concretely defined by the Weyl system. For more details, see {\rm\cite[Section 7]{arve03}}. 

Let $\mathcal{H}=L^2(\mathbb{R})$ and $M=\mathcal{B}(\mathcal{H})$. For ${\bf
x}=(x,y)\in\mathbb{R}^2$, the concrete Weyl operator is $W_{\bf
x}=\exp(\frac{xy}{2}i)U_xV_y$, where $\{U_x\}_{x\in\mathbb{R}}$ and $\{V_x\}_{x\in\mathbb{R}}$ are the unitary groups which have the position operator $Q$ and the momentum operator $P$ as generators, respectively, i.e. 
\[
(U_tf)(x)=e^{itx}f(x),\ (V_tf)(x)=f(x+t)
\]
for $f\in
L^2(\mathbb{R})$ and $t,x\in\mathbb{R}$. Then, the family $\{W_{\bf
x}\}_{{\bf
x}\in\mathbb{R}^2}$ of the unitaries satisfies the Weyl relations
\begin{equation}\label{Weyl}
W_{{\bf
x}_1}W_{{\bf
x}_2}=\exp\left(\frac{i}{2}(x_2y_2-x_1y_2)\right)W_{{\bf
x}_1+{\bf
x}_2}
\end{equation}
for ${\bf
x}_1=(x_1,y_1),{\bf
x}_2=(x_2,y_2)\in\mathbb{R}^2$. The CCR heat flow is defined as the unique CP$_0$-semigroup $T=\{T_t\}_{t\geq0}$ on $M$ satisfying $T_t(W_{\bf
x})=\exp(-t\|x\|^2)W_{\bf
x}$ for all ${\bf
x}\in\mathbb{R}^2$ and $t\geq0$. More precisely, we define $T_t$ for $t\geq0$ by a weak integral $T_t(x)=\int_{\mathbb{R}^2}W_{\frac{{\bf
x}}{\sqrt{2}}}xW_{\frac{{\bf
x}}{\sqrt{2}}}^*d\mu_t({\bf
x})$ for each $x\in
M$, where $\mu_t$ is the probability measure whose Fourier transformation is $u_t({\bf
x})=\exp(-t\|{\bf
x}\|^2)$. 
\end{exam}

According to Stinespring's dilation theorem, for a UCP-map $T$ from a C$^*$-algebra $A$ into $\mathcal{B}(\mathcal{H})$, there exist a Hilbert space $\mathcal{K}$, a unital representation of $A$ on $\mathcal{K}$ and an isometry $v:\mathcal{H}\to\mathcal{K}$ such that $T(a)=v^*\pi(a)v$ for all $a\in
A$. However, Stinespring's theorem does not apply to CP$_0$-semigroup. The notion of dilation of CP$_0$-semigroups are introduced as follows:
\begin{defi}\label{defdilation}
Let $T=\{T_t\}_{t\geq0}$ be a CP$_0$-semigroup on a von Neumann algebra $M$. A dilation of $T$ consists of a von Neumann algebra $N$, a projection $p\in
N$ and an E$_0$-semigroup $\{\theta_t\}_{t\geq0}$ on $N$ such that $M=pNp$ and $T_t(x)=p\theta_t(x)p$ for all $x\in
M$ and $t\geq0$. In addition, if $N$ is generated by $\theta_{[0,\infty)}(M)$ and the central support of $p$ in $N$ is $1_N$, the dilation is said to be minimal.
\end{defi}
Note that a minimal dilation of a CP$_0$-semigroup is unique (if it exists). The existence of minimal dilations is proved by Bhat-Skeide and Muhly-solel. In Section \ref{CP$_0$-semigroups and units of product systems of W$^*$-bimodules}, a relation between the two constructions will be clarified. Arveson also constructed the minimal dilation by other approach in \cite{arve03} (or \cite{arve02}).

The notion of cocycle for E$_0$-semigroups which is useful for the classification of E$_0$-semigroups, is introduced as the following definition.
\begin{defi}
Let $\theta$ be an E$_0$-semigroup on a von Neumann algebra $M$. A family $w=\{w_t\}_{t\geq0}\subset
M$ is called a right cocycle for $\theta$ if $w_{s+t}=\theta_t(w_s)w_t$ for all $s,t\geq0$. If each $w_t$ is unitary {\rm(}contractive{\rm)}, then $w$ is called a right unitary {\rm(}contractive, respectively{\rm)} cocycle.
\end{defi}
\begin{defi}
Two E$_0$-semigroups $\alpha$ and $\beta$ on a von Neumann algebra $M$ are said to be cocycle equivalent if there exists a strongly continuous right unitary cocycle $w$ such that $\beta_t(x)=w_t^*\alpha_t(x)w_t$ for all $t\geq0$ and $x\in
M$. Then, the E$_0$-semigroup $\beta$ is called the cocycle perturbation of $\alpha$ with respect to $w$.

Let $\alpha$ and $\beta$ be E$_0$-semigroups on von Neumann algebras $M$ and $N$, respectively, and $\Phi:M\to
N$ is a $*$-isomorphism. The conjugation $\beta^\Phi$ of $\beta$ with respect to $\Phi$ is an E$_0$-semigroup on $M$ defined by $\beta_t^\Phi=\Phi^{-1}\circ\beta_t\circ\Phi$ for each $t\geq0$. If $\beta^\Phi$ is a cocycle perturbation of $\alpha$, we say $\alpha$ and $\beta$ are cocycle conjugate.
\end{defi}

We will establish a product system of W$^*$-bimodule from a given CP$_0$-semigroup in Section \ref{CP$_0$-semigroups and units of product systems of W$^*$-bimodules}. For this, we prepare a W$^*$-bimodule equipped with a information of a given normal UCP-map as follows. Let $T$ be a normal UCP-map on a von Neumann algebra $M$ and $\mathcal{H}$ a W$^*$-$M$-$N$-bimodule. We define $M\otimes_T\mathcal{H}$ as the completion of the algebraic tensor product $M\otimes_{{\rm
alg}}\mathcal{H}$ with respect to an inner product defined by
\[
\langle
x\otimes\xi,y\otimes\eta\rangle=\langle\xi,T(x^*y)\eta\rangle
\]
for each $x,y\in
M$ and $\xi,\eta\in\mathcal{H}$. The Hilbert space $M\otimes_T\mathcal{H}$ has the canonical W$^*$-$M$-$N$-bimodule structure: $a(x\otimes\xi)b=(ax)\otimes(\xi
b)$ for each $a,x\in
M,\ b\in
N$ and $\xi\in\mathcal{H}$. When $\mathcal{H}=L^2(M)$, we can provide the following formula related to the relative tensor products and normal UCP-maps, which will be used for computing inner products in later arguments.
\begin{prop}\label{formula}
Let $T$ be a normal UCP-map on a von Neumann algebra $M$. For $x,y\in
M$, we have $x\otimes
y\phi^\frac{1}{2}\in\mathcal{D}(M\otimes_TL^2(M);\phi)$. Moreover, we have
\[
\pi_\phi(x_1\otimes
y_1\phi^\frac{1}{2})^*\pi_\phi(x_2\otimes
y_2\phi^\frac{1}{2})=y_1^*T(x_1^*x_2)y_2
\]
for $x_1,x_2,y_1,y_2\in
M$.
\end{prop}
\begin{proof}
For $x',y',z\in
M$, we can compute as
\[
\langle\pi_\phi(x_1\otimes
y_1\phi^\frac{1}{2})\phi^\frac{1}{2}z,x'\otimes
y'\phi^\frac{1}{2}z'\rangle=\langle
x_1\otimes
y_1\phi^\frac{1}{2}z,x'\otimes
y'\phi^\frac{1}{2}z'\rangle=\langle\phi^\frac{1}{2}z,y_1^*T(x_1^*x')y'\phi^\frac{1}{2}z'\rangle,
\]
and hence we have $\pi_\phi(x_1\otimes
y_1\phi^\frac{1}{2})^*(x'\otimes
y'\phi^\frac{1}{2}z')=y_1^*T(x_1^*x')y'\phi^\frac{1}{2}z'$. Thus, we conclude that $\pi_\phi(x_1\otimes
y_1\phi^\frac{1}{2})^*\pi_\phi(x_2\otimes
y_2\phi^\frac{1}{2})\phi^\frac{1}{2}z=y_1^*T(x_1^*x_2)y_2\phi^\frac{1}{2}z.$
\end{proof}

Finally, we prepare notations related with partitions. We fix $t>0$. Let $\mathfrak{P}_t$ be the set of all finite tuples $\mathfrak{p}=(t_1,\cdots,t_n)$ with $t_i>0$ such that $\sum_{i=1}^nt_i=t$. For $\mathfrak{p}=(t_1,\cdots,t_n)\in\mathfrak{P}_t$, we define $\#\mathfrak{p}=n$. Let $\mathfrak{p}=(t_1,\cdots,t_n), \mathfrak{q}=(s_1,\cdots,s_m)\in\mathfrak{P}_t$. We define the joint tuple by $\mathfrak{p}\lor\mathfrak{q}=(t_1,\cdots,t_n,s_1,\cdots,s_m)$ for $\mathfrak{p}=(t_1,\cdots,t_n), \mathfrak{q}=(s_1,\cdots,s_m)\in\mathfrak{P}_t$. Also, we write $\mathfrak{p}\succ\mathfrak{q}$ if there exist partitions $\mathfrak{q}_i\in\mathfrak{P}_{s_i}$ for $i=1,\cdots,m$ such that $\mathfrak{p}=\mathfrak{q}_1\lor\cdots\lor\mathfrak{q}_m$. Let $\mathfrak{P}_0$ be the singleton of the empty tuple $()$ satisfying $\mathfrak{p}\lor()=()\lor\mathfrak{p}=\mathfrak{p}$. Note that when we consider partitions of an interval $[0,t]$, treating $\mathfrak{P}_t$ or the set $\mathfrak{P}_t'$ of all finite tuples $(t_1,\cdots,t_n)$ such that $t=t_n>t_{n-1}>\cdots>t_1>0$ is equivalent because $\mathfrak{P}_t$ and $\mathfrak{P}_t'$ are order isomorphic via a map $\mathfrak{o}:\mathfrak{P}_t\to\mathfrak{P}_t'$ defined by $\mathfrak{o}(t_1,t_2,\cdots,t_n)=\left(\sum_{i=1}^1t_i,\sum_{i=1}^{2}t_i,\cdots,\sum_{i=1}^nt_i\right)$ for each $\mathfrak{p}=(t_1,\cdots,t_n)\in\mathfrak{P}_t$.

\section{Product systems of W$^*$-bimodules and maximal dilations}\label{Product systems of W$^*$-bimodules}
In this section, we will introduce a new concepts of product system of W$^*$-bimodules and their units which are inspired by the definitions of Arveson's and Bhat-Skeide's product systems and units. We will construct an algebraic E$_0$-semigroup from a given unit of a product system of W$^*$-bimodules by taking the inductive limit.

\begin{defi}\label{defrelativesystem}
Let $M$ be a von Neumann algebra and $H=\{\mathcal{H}_t\}_{t\geq0}$ a family of W$^*$-$M$-bimodules with $\mathcal{H}_0=L^2(M)$. If there exist bimodule unitaries $U_{s,t}:\mathcal{H}_s\otimes^M\mathcal{H}_t\to\mathcal{H}_{s+t}$ for each $s,t\geq0$ such that
\begin{equation}\label{productsystemunitary}
U_{r,s+t}({\rm
id}_{\mathcal{H}_r}\otimes^MU_{s,t})=U_{r+s,t}(U_{r,s}\otimes^M{\rm
id}_{\mathcal{H}_t})
\end{equation}
for each $r,s,t\geq0$, and $U_{0,t}$ and $U_{s,0}$ are the canonical identifications, then the pair $(H,\{U_{s,t}\}_{s,t\geq0})$ is called a product system of W$^*$-$M$-bimodules, where the notation $\otimes^M$ denotes the relative tensor product over $M$.
\end{defi}
\begin{rema}
Precisely speaking, associativity {\rm(\ref{productsystemunitary})} means that the following diagram commutes.
\[
\begin{xy}
(75,0) *{\mathcal{H}_{r+s+t}}="D",
(35,0) *{\mathcal{H}_{r+s}\otimes^M\mathcal{H}_t}="E",
(35,-15) *{(\mathcal{H}_r\otimes^M\mathcal{H}_s)\otimes^M\mathcal{H}_t}="F",
(115,-15) *{\mathcal{H}_r\otimes^M(\mathcal{H}_s\otimes^M\mathcal{H}_t)}="G",
(115,0) *{\mathcal{H}_r\otimes^M\mathcal{H}_{s+t}}="H",
\ar "H";"D"^{U_{r,s+t}}
\ar "E";"D"^{U_{r+s,t}}
\ar "F";"E"^{U_{r,s}\otimes{\rm
id}_{\mathcal{H}_t}}
\ar "F";"G"^a
\ar "G";"H"^{{\rm
id}_{\mathcal{H}_r}\otimes
U_{s,t}}
\end{xy}
\]
Here, the morphism $a$ is the associativity isomorphism which is discussed in detail in \cite{sawa-yama17}. By \cite[Theorem 3.2]{sawa-yama17}, we can choose either the left or the right relative tensor product. We will construct a product system of W$^*$-bimodules from a given CP$_0$-semigroup by left relative tensor products.
\end{rema}
\begin{rema}
The notion of product system of W$^*$-bimodules is a direct extension of Arveson's product system. Indeed, if $M=\mathbb{C}$, then a product system of W$^*$-$\mathbb{C}$-bimodules is just an Arveson's product system without a measurable structure. For the purpose of this paper, measurable structures for product systems of W$^*$-bimodules are not necessary.
\end{rema}
\begin{defi}\label{units}
Let $M$ be a von Neumann algebra with a faithful normal state $\phi$ and $(H,\{U_{s,t}\}_{s,t\geq0})$ a product system of W$^*$-$M$-bimodules. A family $\Xi=\{\xi(t)\}_{t\geq0}$ of $\xi(t)\in\mathcal{D}(\mathcal{H}_t;\phi)$ is called a unit of $H$ with respect to $\phi$ if $\xi(0)=\phi^\frac{1}{2}$ and
\[
U_{s,t}(\xi(s)\phi^{-\frac{1}{2}}\xi(t))=\xi(s+t)
\]
for all $s,t\geq0$.

If a unit $\Xi=\{\xi(t)\}_{t\geq0}$ satisfies $\pi_\phi(\xi(t))^*\pi_\phi(\xi(t))=1_M\ (\|\pi_\phi(\xi(t))^*\pi_\phi(\xi(t))\|\leq1)$, it is said to be unital {\rm(}contractive, respectively{\rm)}.
\end{defi}
We introduce a natural notion of isomorphism between product systems of W$^*$-$M$-bimodules and generating property for units as follows:
\begin{defi}\label{isomorphismsystem}
Let $(H,\{U_{s,t}\}_{s,t\geq0})$ and $(K,\{V_{s,t}\}_{s,t\geq0})$ be  product systems of W$^*$-$M$-bimodules. An isomorphism is a family $u^\otimes=\{u_t\}_{t\geq0}$ of $M$-bilinear unitaries $u_t:\mathcal{H}_t\to\mathcal{K}_t$ satisfying 
\begin{eqnarray}\label{isomorphism}
V_{s,t}(u_s\otimes^M
u_t)=u_{s+t}U_{s,t}
\end{eqnarray}
for all $s,t\geq0$. Then, the product system $H$ is said to be isomorphic to $K$ and we denote as $H\cong
K$.
\end{defi}
\begin{defi}
Let $H=\{\mathcal{H}_t\}_{t\geq0}$ be a product system of W$^*$-$M$-bimodules with $M$-bilinear unitaries $\{U_{s,t}\}_{s,t\geq0}$. For $t\geq0$ and $\mathfrak{p}=(t_1,\cdots,t_n)\in\mathfrak{P}_t$, we denote the $M$-bilinear unitary 
\begin{eqnarray*}
U(\mathfrak{p})=U_{t_1,t_1'}({\rm
id}_{t_1}\otimes
U_{t_2,t_2'})({\rm
id}_{t_1,t_2}\otimes
U_{t_3,t_3'})\cdots({\rm
id}_{t_1,\cdots,t_{n-2}}\otimes
U_{t_{n-1},t_{n-1}'})
\end{eqnarray*}
from $\mathcal{H}_{t_1}\otimes^M\cdots\otimes^M\mathcal{H}_{t_n}$ onto $\mathcal{H}_{t}$, where $t_i'=t_{i+1}+\cdots+t_{n}$ and ${\rm
id}_{t_1,\cdots,t_{i}}={\rm
id}_{t_1}\otimes\cdots\otimes{\rm
id}_{t_i}$. A unital unit $\Xi=\{\xi(t)\}_{t\geq0}$ with respect to $\phi$ is said to be generating when the set 
\[
\{U(\mathfrak{p})(x_1\xi(t_1)\phi^{-\frac{1}{2}}\cdots
\phi^{-\frac{1}{2}}x_{n-1}\xi(t_{n-1})\phi^{-\frac{1}{2}}x_n\xi(t_n)y)\mid\mathfrak{p}\in\mathfrak{P}_t,\ x_1,\cdots,x_n,y\in
M\}
\]
is dense in $\mathcal{H}_{t}$ for all $t\geq0$.
\end{defi}
\begin{rema}
Like Arveson's and Bhat-Skeide's product system, every product system of W$^*$-$M$-bimodules does not always have a unit with respect to a faithful normal state $\phi$ on $M$. We say that a product system of W$^*$-$M$-bimodules equipped with a unit with respect to $\phi$ is $\phi$-spatial. From now, we fix a faithful normal state $\phi$ on a von Neumann algebra $M$ through this paper . When we say a unit merely, suppose that it is a unit with respect to $\phi$ of a $\phi$-spatial product system.
\end{rema}
Let $\Xi$ be a unital unit of a $\phi$-spatial product system $H$ of W$^*$-$M$-bimodules. We shall define the inductive limit $\mathfrak{H}$ of $H$ which depends on $\Xi$ and an algebraic E$_0$-semigroup on ${\rm
End}(\mathfrak{H}_M)$ as follows: for $0\leq
s\leq
t$, we define a right $M$-linear isometry $b_{t,s}:\mathcal{H}_s\to\mathcal{H}_t$ by
\begin{equation}\label{inductivebeta}
b_{t,s}(\xi)=U_{t-s,s}(\xi(t-s)\phi^{-\frac{1}{2}}\xi)
\end{equation}
for each $\xi\in\mathcal{H}_s$. Note that $b_{t,s}\circ
b_{s,r}=b_{t,r}$ for $0\leq
r\leq
s\leq
t$. Let $\mathfrak{H}$ be the inductive limit of the inductive system $(H,\{b_{t,s}\}_{s\leq
t})$ and $\kappa_t:\mathcal{H}_t\to\mathfrak{H}$ the canonical embedding for each $t\geq0$. The right W$^*$-$M$-module $\mathfrak{H}$ is called the inductive limit of the pair $(H,\Xi)$.
\begin{theo}\label{U_tgeneral}
Fix $t\geq0$. There exists a right $M$-linear unitary $U_t:\mathfrak{H}\otimes^M\mathcal{H}_t\to\mathfrak{H}$
\end{theo}
\begin{proof}
For $s\geq0$, $\xi_s\in\mathcal{D}(\mathcal{H}_t;\phi)$ and $\eta_t\in\mathcal{H}_t$, we define
\[
U_t((\kappa_s\xi_s)\phi^{-\frac{1}{2}}\eta_t)=\kappa_{s+t}U_{s,t}(\xi_s\phi^{-\frac{1}{2}}\eta_t).
\]
We shall show that $U_t$ is an isometry. For $s\geq0$, $\xi_s,\xi_s'\in\mathcal{H}_s$ and $\eta_t,\eta_t'\in\mathcal{H}_t$, we have
\begin{eqnarray*}
&&\hspace{-20pt}\langle
U_t((\kappa_s\xi_s)\phi^{-\frac{1}{2}}\eta_t),U_t((\kappa_s\xi_s')\phi^{-\frac{1}{2}}\eta_t')\rangle=\langle\xi_s\phi^{-\frac{1}{2}}\eta_t,\xi_s'\phi^{-\frac{1}{2}}\eta_t'\rangle\\
&&\hspace{-20pt}=\langle\eta_t,\pi_\phi(\xi_s)^*\pi_\phi(\xi_s')\eta_t'\rangle=\langle\eta_t,\pi_\phi(\kappa_s\xi_s)^*\pi_\phi(\kappa_s\xi_s')\eta_t'\rangle=\langle(\kappa_s\xi_s)\phi^{-\frac{1}{2}}\eta_t,(\kappa_s\xi_s')\phi^{-\frac{1}{2}}\eta_t'\rangle.
\end{eqnarray*}
This implies that for $s\geq0$, $\xi_s\in\mathcal{H}_s,\ \zeta_r\in\mathcal{H}_r$ and $\eta_t,\eta_t'\in\mathcal{H}_t$, in the general case, we have
\begin{eqnarray*}
&&\langle
U_t((\kappa_s\xi_s)\phi^{-\frac{1}{2}}\eta_t),U_t((\kappa_r\zeta_r)\phi^{-\frac{1}{2}}\eta_t')\rangle\\
&&=\langle
U_t((\kappa_{s+t}b_{s+r.s}\xi_s)\phi^{-\frac{1}{2}}\eta_t),U_t((\kappa_{s+r}b_{s+r,r}\zeta_r)\phi^{-\frac{1}{2}}\eta_t')\rangle\\
&&=\langle
(\kappa_{s+t}b_{s+r.s}\xi_s)\phi^{-\frac{1}{2}}\eta_t,(\kappa_{s+r}b_{s+r,r}\zeta_r)\phi^{-\frac{1}{2}}\eta_t'\rangle=\langle
(\kappa_s\xi_s)\phi^{-\frac{1}{2}}\eta_t,(\kappa_r\zeta_r)\phi^{-\frac{1}{2}}\eta_t'\rangle.
\end{eqnarray*}

We shall check that $U_t$ is surjective. In the case when $s\leq
t$, for $\eta=\kappa_s\eta_s\in\mathfrak{H}$, we can conclude that the image of $(\kappa_0\phi^{\frac{1}{2}})\phi^{-\frac{1}{2}}\kappa_t^*\kappa_s\eta_s$ by $U_t$ is $\eta$. In the case when $s>t$, let $\mathcal{D}$ be a subspace of $\mathcal{H}_{s-t}\otimes^M\mathcal{H}_t$ spanned by vectors $\eta_{s-t}\phi^{-\frac{1}{2}}\eta_t$ for all $\phi$-bounded vectors $\eta_{s-t}\in\mathcal{H}_{s-t}$ and $\eta_t\in\mathcal{H}_t$. For $\eta=\kappa_s\eta_s\in\mathfrak{H}$, $\eta_s$ can be approximated by vectors $U_{s-t,t}\zeta$ for some $\zeta\in\mathcal{D}$ and we have $U_t(\kappa_{s-t}\otimes{\rm
id}_{\mathcal{H}_t})=\kappa_sU_{s-t,t}$ on $\mathcal{D}$.
\end{proof}

The von Neumann algebra $M$ can be represented faithfully on $\mathfrak{H}$ by 
\begin{equation}\label{faithfulrepM}
\pi(x)\xi=\kappa_0(x(\kappa_0^*\xi))
\end{equation}
for each $x\in
M$ and $\xi\in\mathfrak{H}$. Note that $\pi(M)\subset{\rm
End}(\mathfrak{H}_M)$. Then, we can define an algebraic E$_0$-semigroup $\theta$ on the von Neumann algebra ${\rm
End}(\mathfrak{H}_M)$ by
\begin{equation}\label{maxdilation}
\theta_t(a)=U_t(a\otimes^M{\rm
id}_{\mathcal{H}_t})U_t^*
\end{equation}
for each $a\in{\rm
End}(\mathfrak{H}_M)$. The algebraic E$_0$-semigroup $\theta$ is called the dilation of the pair $(H,\Xi)$. If we assume a continuity to the unit $\Xi$, then it becomes an E$_0$-semigroup as follows:
\begin{prop}
If the unital unit $\Xi=\{\xi(t)\}_{t\geq0}$ of the $\phi$-spatial product system $H$ of W$^*$-$M$-bimodules satisfies that
\begin{equation}\label{continuity2}
U_t(\xi\phi^{-\frac{1}{2}}\xi(t))\to\xi\hspace{20pt}(t\to+0)
\end{equation}
for all $\xi\in\mathcal{D}(\mathfrak{H};\phi)$, then the dilation $\theta$ of the pair $(H,\Xi)$ is an E$_0$-semigroup.
\end{prop}
\begin{proof}
For each $a\in{\rm
End}(\mathfrak{H}_M)$ and each $\phi$-bounded vector $\xi\in\mathfrak{H}$, we have
\begin{eqnarray*}
\theta_t(a)\xi-a\xi&=&\theta_t(a)\xi-U_t(a\xi\phi^{-\frac{1}{2}}\xi(t))+U_t(a\xi\phi^{-\frac{1}{2}}\xi(t))-a\xi\\
&=&\theta_t(a)(\xi-U_t(\xi\phi^{-\frac{1}{2}}\xi(t)))+U_t(a\xi\phi^{-\frac{1}{2}}\xi(t))-a\xi\to0
\end{eqnarray*}
when $t\to0$. Thus, the map $t\mapsto\theta_t(a)$ is $\sigma$-weakly continuous for each $a\in{\rm
End}(\mathfrak{H}_M)$.
\end{proof}
\begin{defi}
We say that a unit $\Xi$ with {\rm(\ref{continuity2})} is continuous.
\end{defi}
A right cocycle $w=\{w_t\}_{t\geq0}$ for $\theta$ is said to be adapted if $\kappa_t\kappa_t^*w_t\kappa_t\kappa_t^*=w_t$ for all $t\geq0$, where $\kappa_t$ is the canonical embedding from $\mathcal{H}_t$ into $\mathcal{H}$. The following theorem describing a correspondence between cocycles and units, will enable as to classify E$_0$-semigroups by product systems of W$^*$-bimodules in Section \ref{Classification of E$_0$-semigroups}.
\begin{theo}\label{unitcocycle}
Let $\theta=\{\theta_t\}_{t\geq0}$ be the dilation of a pair $(H,\Xi)$ of a $\phi$-spatial product system $H=\{\mathcal{H}_t\}_{t\geq0}$ of W$^*$-$M$-bimodules and a continuous unital unit $\Xi=\{\xi(t)\}_{t\geq0}$. There exists a one-to-one correspondence between contractive adapted right cocycles $w=\{w_t\}_{t\geq0}$ on ${\rm
End}(\mathfrak{H}_M)$ and contractive units $\Lambda=\{\lambda(t)\}_{t\geq0}$ of $H$ by relations $\lambda(t)=\kappa_t^*w_t\kappa_0\phi^\frac{1}{2}$ and $w_t=\pi_\phi(\kappa_t\lambda(t))\pi_\phi(\kappa_0\phi^\frac{1}{2})^*$ for all $t\geq0$.
\end{theo}
\begin{proof}
Let $\{U_{s,t}\}_{s,t\geq0}$ be a family giving the relative product system structure of $H$ and $\mathfrak{H}$ the inductive limit of $(H,\Xi)$.

Let $w=\{w_t\}_{t\geq0}$ be a contractive adapted right cocycle for $\theta$. Note that each $\lambda(t)=\kappa_t^*w_t\kappa_0\phi^\frac{1}{2}$ is $\phi$-bounded. Moreover, for each $t\geq0$, we have
\begin{eqnarray}\label{etat}
\pi_\phi(\lambda(t))^*\pi_\phi(\lambda(t))\phi^\frac{1}{2}x=\kappa_0^*w_t^*w_t\kappa_0\phi^\frac{1}{2}x
\end{eqnarray}
for each $x\in
M$, and hence the contractivity of $w_t$ implies that $\|\pi_\phi(\lambda(t))^*\pi_\phi(\lambda(t))\|\leq1$. We shall show that $\Lambda$ is a unit. For $s,t\geq0$, $\kappa_{s+t}=U_t(\kappa_s\otimes{\rm
id}_t)U_{s,t}^*$ implies the following calculations.
\begin{eqnarray*}
\lambda(s+t)&=&\kappa_{s+t}^*w_{s+t}\kappa_0\phi^\frac{1}{2}=\kappa_{s+t}^*\theta_t(w_s)w_t\kappa_0\phi^\frac{1}{2}=\kappa_{s+t}^*U_t(w_s\otimes{\rm
id}_{t})U_t^*w_t\kappa_0\phi^\frac{1}{2}\\
&=&\kappa_{s+t}^*U_t(w_s\otimes{\rm
id}_{t})((\kappa_0\phi^\frac{1}{2})\phi^{-\frac{1}{2}}(\kappa_t^*w_t\kappa_0\phi^\frac{1}{2}))=\kappa_{s+t}^*U_t((w_s\kappa_0\phi^\frac{1}{2})\phi^{-\frac{1}{2}}(\kappa_t^*w_t\kappa_0\phi^\frac{1}{2}))\\
&=&U_{s,t}(\kappa_s^*\otimes{\rm
id}_t)U_t^*U_t((w_s\kappa_0\phi^\frac{1}{2})\phi^{-\frac{1}{2}}(\kappa_t^*w_t\kappa_0\phi^\frac{1}{2}))=U_{s,t}(\lambda(s)\phi^{-\frac{1}{2}}\lambda(t)).
\end{eqnarray*}

Conversely, let $\Lambda=\{\lambda(t)\}_{t\geq0}$ be a contractive unit of $H$, and for each $t\geq0$, $w_t=\pi_\phi(\kappa_t\lambda(t))\pi_\phi(\kappa_0\phi^\frac{1}{2})^*\in{\rm
End}(\mathcal{H}_M)$. For all $\xi\in\mathcal{H}$, the equation
\begin{eqnarray}\label{wt}
w_t\xi=\kappa_tU_{t,0}(\lambda(t)\phi^{-\frac{1}{2}}\kappa_0^*\xi)
\end{eqnarray}
is implied from the approximation of $\kappa_0^*\xi$ by vectors as the form of $\phi^\frac{1}{2}x$. In particular, 
\begin{eqnarray}\label{wt0}
w_t(\kappa_0\phi^\frac{1}{2}x)=\kappa_t\lambda(t)x
\end{eqnarray}
for all $t\geq0,\ x\in
M$, and hence $w_t=0$ on the orthogonal complement of the closed subspace $\kappa_0\kappa_0^*\mathcal{H}$. Thus, computations
\begin{eqnarray*}
\theta_t(w_s)w_t(\kappa_0\phi^\frac{1}{2}x)&=&U_t(w_s\otimes{\rm
id}_t)U_t^*\pi_\phi(\kappa_t\lambda(t))\pi_\phi(\kappa_0\phi^\frac{1}{2})^*\kappa_0(\phi^\frac{1}{2}x)\\
&=&U_t(w_s\otimes{\rm
id}_t)U_t^*\kappa_t(\lambda(t)x)=U_t(w_s\otimes{\rm
id}_t)((\kappa_0\phi^\frac{1}{2})\phi^{-\frac{1}{2}}\lambda(t))\\
&=&U_t((w_s\kappa_0\phi^\frac{1}{2})\phi^{-\frac{1}{2}}(\eta_tx))=U_t((\kappa_s\lambda(s))\phi^{-\frac{1}{2}}(\lambda(t)x))\\
&=&\kappa_{s+t}U_{s,t}(\lambda(s)\phi^{-\frac{1}{2}}(\lambda(t)x))=w_{s+t}(\kappa_0\phi^\frac{1}{2}x)
\end{eqnarray*}
for every $x\in
M$, implies that $w$ is a right cocycle. We shall show that $w$ is adapted. For all $t\geq0$ and all $\xi\in\mathcal{H}$, by (\ref{wt}), we have
\begin{eqnarray*}
\kappa_t\kappa_t^*w_t\xi=\kappa_t\kappa_t^*w_t\kappa_0\kappa_0^*\xi=\kappa_t\kappa_t^*\kappa_tU_{t,0}(\lambda(t)\phi^{-\frac{1}{2}}(\kappa_0^*\xi))=w_t\kappa_0\kappa_0^*\xi=w_t\xi.
\end{eqnarray*}
By (\ref{wt}) again and the fact that the family $\{\kappa_t\kappa_t^*\}_{t\geq0}$ is increasing,  we have also
\[
w_t\kappa_t\kappa_t^*\xi=\kappa_tU_{t,0}(\lambda(t)\phi^{-\frac{1}{2}}\kappa_0^*\kappa_t\kappa_t^*\xi)\\
=\kappa_tU_{t,0}(\lambda(t)\phi^{-\frac{1}{2}}\kappa_0^*\xi)=w_t\xi.
\]
We conclude that $\kappa_t\kappa_t^*w_t\kappa_t\kappa_t^*=w_t$, that is, the adaptedness.

We can check that the correspondence between contractive adapted right cocycles and contractive units is one-to-one by (\ref{wt}). 
\end{proof}
Note that by (\ref{etat}) and (\ref{wt}), the unit associated with an adapted unitary right cocycle is unital, and the contractive adapted right cocycle associated with a unital unit preserves inner products on $\kappa_0\kappa_0^*\mathcal{H}$.

\section{CP$_0$-semigroups and units of product systems of W$^*$-bimodules}\label{CP$_0$-semigroups and units of product systems of W$^*$-bimodules}
In this section, we will obtain a one-to-one correspondence between algebraic CP$_0$-semigroups on $M$ and pairs of ($\phi$-spatial) product systems of W$^*$-$M$-bimodules and generating unital units up to unit preserving isomorphism. It will be shown that the dilation of the pair associated with a given CP$_0$-semigroup $T$ is a dilation of $T$. Also, we will discuss a relation between the continuity of CP$_0$-semigroups and one of units as follows: the unit associated with a CP$_0$-semigroup is continuous, and conversely, a continuous unit gives rise to a CP$_0$-semigroup.

First, we shall establish an algebraic CP$_0$-semigroup from a unit. Let $\Xi=\{\xi(t)\}_{t\geq0}$ be a unital unit of a $\phi$-spatial product system $H=\{\mathcal{H}_t\}_{t\geq0}$ of W$^*$-$M$-bimodules. We define a unital linear map $T_t^{\Xi}$ on $M$ by
\begin{equation}\label{cpfromunit}
T_t^{\Xi}(x)=\pi_\phi(\xi(t))^*\pi_\phi(x\xi(t))\in
M
\end{equation}
for each $t\geq0$ and $x\in
M$.
\begin{lemm}\label{CP0-semigroups}
The family $T^{\Xi}=\{T_t^{\Xi}\}_{t\geq0}$ is an algebraic CP$_0$-semigroup.
\end{lemm}
\begin{proof}
By the definition, it is clear that each $T_t^{\Xi}$ is normal completely positive map.

For $s,t\geq0$ and $x,y,z\in
M$, we can compute as
\begin{eqnarray*}
&&\langle
T_s^{\Xi}(T_t^{\Xi}(x))\phi^\frac{1}{2}y,\phi^\frac{1}{2}z\rangle=\langle\pi_\phi(\xi(s))^*\pi_\phi(\pi_\phi(\xi(t))^*\pi(x\xi(t))\xi(s))\phi^\frac{1}{2}y,\phi^\frac{1}{2}z\rangle\\
&&=\langle\pi_\phi(\xi(t))^*\pi(x\xi(t))\xi(s)y,\xi(s)z\rangle=\langle
x\xi(t)\phi^{-\frac{1}{2}}\xi(s)y,\xi(t)\phi^{-\frac{1}{2}}\xi(s)z\rangle\\
&&=\langle
xU_{s,t}(\xi(t)\phi^{-\frac{1}{2}}\xi(s))y,U_{s,t}(\xi(t)\phi^{-\frac{1}{2}}\xi(s)z)\rangle=\langle
x\xi(t+s)y,\xi(t+s)z\rangle\\
&&=\langle
\pi_\phi(\xi(s+t))^*\pi_\phi(x\xi(s+t))\phi^\frac{1}{2}y,\phi^\frac{1}{2}z\rangle=\langle
T_{s+t}^{\Xi}(x)\phi^\frac{1}{2}y,\phi^\frac{1}{2}z\rangle,
\end{eqnarray*}
and hence $T_s^{\Xi}\circ
T_t^{\Xi}=T_{s+t}^{\Xi}$.
\end{proof}

We can describe the continuity for the algebraic CP$_0$-semigroup $T^{\Xi}$ as the one for the unit $\Xi$ as the following theorem. 
\begin{theo}\label{condition for continuity}
Let $\mathfrak{H}$ be the inductive limit of the pair $(H,\Xi)$ and $U_t:\mathfrak{H}\otimes^M\mathcal{H}_t\to\mathfrak{H}$ the unitary in {\rm
Theorem \ref{U_tgeneral}}. The semigroup $T^{\Xi}$ associated with $\Xi$ is a CP$_0$-semigroup if and only if 
\begin{eqnarray}\label{continuity}
U_t(\kappa_0(x\phi^\frac{1}{2})\phi^{-\frac{1}{2}}\xi(t))=\kappa_t(x\xi(t))\to\kappa_0(x\phi^\frac{1}{2})\hspace{20pt}(t\to+0)
\end{eqnarray}
holds for each $x\in
M$.
\end{theo}
\begin{proof}
Suppose (\ref{continuity}) for all $x\in
M$. For $t\geq0$ and $x,y,z\in
M$, we have
\begin{eqnarray*}
\langle
T_t^{\Xi}(x)\phi^\frac{1}{2}y,\phi^\frac{1}{2}z\rangle&=&\langle
\pi_\phi(\xi(t))^*\pi_\phi(x\xi(t))\phi^\frac{1}{2}y,\phi^\frac{1}{2}z\rangle=\langle
x\xi(t)y,\xi(t)z\rangle\\
&=&\langle
U_t(\kappa_0(x\phi^\frac{1}{2})\phi^{-\frac{1}{2}}\xi(t)y),U_t(\kappa_0(\phi^\frac{1}{2})\phi^{-\frac{1}{2}}\xi(t)z)\rangle.
\end{eqnarray*}
Thus, when $t\to+0$, the inner product $\langle
T_t(x)\phi^\frac{1}{2}y,\phi^\frac{1}{2}z\rangle$ tends to $\langle\kappa_0(x\phi^\frac{1}{2})y,\kappa_0(\phi^\frac{1}{2})z\rangle=\langle
x\phi^\frac{1}{2}y,\phi^\frac{1}{2}z\rangle.$ We conclude that for every $x\in
M$, $T_t^{\Xi}(x)\to
x$ weakly when $t\to+0$, and hence $T^{\Xi}$ is a CP$_0$-semigroup by the boundedness of $\{\|T_t^{\Xi}(x)\|\}_{t\geq0}$.

Conversely, we assume that $T^{\Xi}$ is a CP$_0$-semigroup. We can compute as
\begin{eqnarray*}
&&\langle
x\xi(t),x\xi(t)\rangle=\langle
xU_{t,0}(\xi(t)\phi^{-\frac{1}{2}}\phi^\frac{1}{2}),xU_{t,0}(\xi(t)\phi^{-\frac{1}{2}}\phi^\frac{1}{2})\rangle=\langle\xi(t)\phi^{-\frac{1}{2}}\phi^\frac{1}{2},x^*x\xi(t)\phi^{-\frac{1}{2}}\phi^\frac{1}{2}\rangle\\
&&=\langle\phi^\frac{1}{2},\pi_\phi(\xi(t))^*\pi_\phi(x^*x\xi(t))\phi^\frac{1}{2}\rangle=\langle\phi^\frac{1}{2},T_t(x^*x)\phi^\frac{1}{2}\rangle,\\
&&\langle
\kappa_t(x\xi(t)),\kappa_0(x\phi^\frac{1}{2})\rangle=\langle
\kappa_t(x\xi(t)),\kappa_t
b_{t,0}\kappa_0(x\phi^\frac{1}{2})\rangle=\langle
x\xi(t),U_{t,0}(\xi(t)\phi^{-\frac{1}{2}}\kappa_0(x\phi^\frac{1}{2}))\rangle\\
&&=\langle
U_{t,0}(x\xi(t)\phi^{-\frac{1}{2}}\phi^\frac{1}{2}),U_{t,0}(\xi(t)\phi^{-\frac{1}{2}}\kappa_0(x\phi^\frac{1}{2}))\rangle=\langle\phi^\frac{1}{2},T_t(x^*)x\phi^\frac{1}{2}\rangle.
\end{eqnarray*}
Thus, when $t\to+0$, we have $\|\kappa_t(x\xi(t))-\kappa_0(x\phi^\frac{1}{2})\|^2\to0$.
\end{proof}
\begin{defi}
We say that a unit $\Xi$ with {\rm(\ref{continuity})} is weakly continuous.
\end{defi}

Next, we shall construct a product system and a unit from a given algebraic CP$_0$-semigroup $T$ on $M$. For $t>0$ and $\mathfrak{p}=(t_1,\cdots,t_n)\in\mathfrak{P}_t$, we define a W$^*$-$M$-bimodule
\begin{equation}\label{H^T(p,t)}
\mathcal{H}^T(\mathfrak{p},t)=(M\otimes_{t_1}L^2(M))\otimes^M\cdots\otimes^M(M\otimes_{t_n}L^2(M)),
\end{equation}
where $M\otimes_{s}L^2(M)=M\otimes_{T_{s}}L^2(M)$ for each $s\geq0$. Suppose $\mathfrak{p}=(t_1,\cdots,t_n)\succ\mathfrak{q}=(s_1,\cdots,s_m)$ in $\mathfrak{P}_t$, $\mathfrak{p}=\mathfrak{q}(s_1)\lor\cdots\lor\mathfrak{q}(s_m)$ and $\mathfrak{q}(s_i)=(s_{i,1},\cdots,s_{i,k(i)})\in\mathfrak{P}_{s_i}$. We define a map $ a_{\mathfrak{q}(s_i)}:M\otimes_{s_i}L^2(M)\to\mathcal{H}^T(\mathfrak{q}(s_i),s_i)$ by
\[
 a_{\mathfrak{q}(s_i)}(x\otimes_{s_i}y\phi^{\frac{1}{2}})=(x\otimes_{s_{i,1}}\phi^\frac{1}{2})\phi^{-\frac{1}{2}}(1_M\otimes_{s_{i,2}}\phi^\frac{1}{2})\phi^{-\frac{1}{2}}\cdots\phi^{-\frac{1}{2}}(1_M\otimes_{s_{i,k(i)-1}}\phi^\frac{1}{2})\phi^{-\frac{1}{2}}(1_M\otimes_{s_{i,k(i)}}y\phi^\frac{1}{2})
\]
for each $x,y\in
M$. We can check that $ a_{\mathfrak{p}(s_i)}$ is an $M$-bilinear isometry by Proposition \ref{formula}. We define an isometry 
\begin{eqnarray}\label{isometries}
 a_{\mathfrak{p},\mathfrak{q}}= a_{\mathfrak{q}(s_1)}\otimes^M\cdots\otimes^M a_{\mathfrak{q}(s_m)}:\mathcal{H}^T(\mathfrak{q},t)\to\mathcal{H}^T(\mathfrak{p},t).
\end{eqnarray}
Then, the pair $(\{\mathcal{H}^T(\mathfrak{p},t)\}_{\mathfrak{p}\in\mathfrak{P}_t},\{ a_{\mathfrak{p},\mathfrak{q}}\}_{\mathfrak{p}\succ\mathfrak{q}})$ is an inductive system of W$^*$-$M$-bimodules. Let $\mathcal{H}_t^T$ be the inductive limit and $\kappa_{\mathfrak{p},t}:\mathcal{H}^T(\mathfrak{p},t)\to\mathcal{H}_t^T$ the canonical embedding. Put $\mathcal{H}_0^T=L^2(M)$. Also, we define a family $\Xi^T=\{\xi^T(t)\}_{t\geq0}$ by $\xi^T(t)=\kappa_{(t),t}(1_M\otimes\phi^\frac{1}{2})$ for each $t>0$ and $\xi^T(0)=\phi^\frac{1}{2}$.

\begin{theo}\label{relative productH}
The family $H^T=\{\mathcal{H}_t^T\}_{t\geq0}$ is a $\phi$-spatial product system of W$^*$-$M$-bimodules and $\Xi^T$ is a generating unital unit of $H^T$.
\end{theo}
\begin{proof}
For $s,t>0$, we define a map $U_{s,t}^T:\mathcal{H}_s^T\otimes^M\mathcal{H}_t^T\to\mathcal{H}_{s+t}^T$ by 
\[
U_{s,t}^T((\kappa_{\mathfrak{q},s}\xi_{\mathfrak{q}})\phi^{-\frac{1}{2}}(\kappa_{\mathfrak{p},t}\eta_{\mathfrak{q}}))=\kappa_{\mathfrak{q}\lor\mathfrak{p},s+t}(\xi_\mathfrak{q}\phi^{-\frac{1}{2}}\eta_{\mathfrak{p}})
\]
for each $\mathfrak{q}=(s_1,\cdots,s_m)\in\mathfrak{P}_s, \mathfrak{p}=(t_1,\cdots,t_n)\in\mathfrak{P}_t,\ \xi_\mathfrak{q}\in\mathcal{D}(\mathcal{H}^T(\mathfrak{q},s);\phi)$ and $\eta_\mathfrak{p}\in\mathcal{H}^T(\mathfrak{p},t)$
Here, note that $\kappa_{\mathfrak{q},s}\xi_\mathfrak{q}$ is $\phi$-bounded. We shall show that $U_{s,t}^T$ is an isometry, i.e. the equation
\[
\langle(\kappa_{\mathfrak{q},s}\xi_{\mathfrak{q}})\phi^{-\frac{1}{2}}(\kappa_{\mathfrak{p},t}\eta_{\mathfrak{p}}),(\kappa_{\mathfrak{q'},s}\xi_{\mathfrak{q'}}')\phi^{-\frac{1}{2}}(\kappa_{\mathfrak{p'},t}\eta_{\mathfrak{p}'}')\rangle=\langle\kappa_{\mathfrak{q}\lor\mathfrak{p},s+t}(\xi_\mathfrak{q}\phi^{-\frac{1}{2}}\eta_{\mathfrak{p}}),\kappa_{\mathfrak{q'}\lor\mathfrak{p'},s+t}(\xi_\mathfrak{q'}'\phi^{-\frac{1}{2}}\eta_{\mathfrak{p'}}')\rangle
\]
holds for all $\mathfrak{q},\mathfrak{q}'\in\mathfrak{P}_s,\ \mathfrak{p},\mathfrak{p}'\in\mathfrak{P}_t,\ \xi_\mathfrak{q}\in\mathcal{D}(\mathcal{H}^T(\mathfrak{q},s);\phi),\xi_{\mathfrak{q'}}'\in\mathcal{D}(\mathcal{H}^T(\mathfrak{q}',s);\phi),\eta_\frak{p}\in\mathcal{H}^T(\mathfrak{p},t)$ and $\eta_{\mathfrak{p}'}'\in\mathcal{H}^T(\mathfrak{p}',t)$. If $\mathfrak{q}=\mathfrak{q}'$ and $\mathfrak{p}=\mathfrak{p}'$, we have
\begin{eqnarray*}
&&\langle(\kappa_{\mathfrak{q},s}\xi_{\mathfrak{q}})\phi^{-\frac{1}{2}}(\kappa_{\mathfrak{p},t}\eta_{\mathfrak{p}}),(\kappa_{\mathfrak{q},s}\xi_{\mathfrak{q}}')\phi^{-\frac{1}{2}}(\kappa_{\mathfrak{p},t}\eta_{\mathfrak{p}}')\rangle=\langle\kappa_{\mathfrak{p},t}\eta_{\mathfrak{p}},\pi_\phi(\kappa_{\mathfrak{q},s}\xi_{\mathfrak{q}})^*\pi_\phi(\kappa_{\mathfrak{q},s}\xi_{\mathfrak{q}}')\kappa_{\mathfrak{p},t}\eta_{\mathfrak{p}}'\rangle\\
&&=\langle\kappa_{\mathfrak{p},t}\eta_{\mathfrak{p}},\kappa_{\mathfrak{p},t}(\pi_\phi(\kappa_{\mathfrak{q},s}\xi_{\mathfrak{q}})^*\pi_\phi(\kappa_{\mathfrak{q},s}\xi_{\mathfrak{q}}'))\eta_{\mathfrak{p}}')\rangle=\langle\eta_{\mathfrak{p}},\pi_\phi(\kappa_{\mathfrak{q},s}\xi_{\mathfrak{q}})^*\pi_\phi(\kappa_{\mathfrak{q},s}\xi_{\mathfrak{q}}')\eta_{\mathfrak{p}}'\rangle\\
&&=\langle\eta_{\mathfrak{p}},\pi_\phi(\xi_{\mathfrak{q}})^*\pi_\phi(\xi_{\mathfrak{q}}')\eta_{\mathfrak{p}}')\rangle=\langle\kappa_{\mathfrak{q}\lor\mathfrak{p},s+t}(\xi_\mathfrak{q}\phi^{-\frac{1}{2}}\eta_{\mathfrak{p}}),\kappa_{\mathfrak{q}\lor\mathfrak{p},s+t}(\xi_\mathfrak{q}'\phi^{-\frac{1}{2}}\eta_{\mathfrak{p}}')\rangle.
\end{eqnarray*}
In general case, since $\kappa_{\mathfrak{s},s}=\kappa_{\mathfrak{s}',s} a_{\mathfrak{s}',\mathfrak{s}}$ for all $\mathfrak{s},\mathfrak{s}'\in\mathfrak{P}_s$ with $\mathfrak{s}'\succ\mathfrak{s}$, if we take $\hat{\mathfrak{q}}\in\mathfrak{P}_s, \hat{\mathfrak{p}}\in\mathfrak{P}_t,$ such that $\hat{\mathfrak{q}}\succ\mathfrak{q},\mathfrak{q}'$ and $\hat{\mathfrak{p}}\succ\mathfrak{p},\mathfrak{p}'$, then
\begin{eqnarray*}
&&\langle(\kappa_{\mathfrak{q},s}\xi_{\mathfrak{q}})\phi^{-\frac{1}{2}}(\kappa_{\mathfrak{p},t}\eta_{\mathfrak{p}}),(\kappa_{\mathfrak{q'},s}\xi_{\mathfrak{q'}}')\phi^{-\frac{1}{2}}(\kappa_{\mathfrak{p'},t}\eta_{\mathfrak{p}'}')\rangle\\
&&=\langle(\kappa_{\hat{\mathfrak{q}},s} a_{\hat{\mathfrak{q}},\mathfrak{q}}\xi_{\mathfrak{q}})\phi^{-\frac{1}{2}}(\kappa_{\hat{\mathfrak{p}},t} a_{\hat{\mathfrak{p}},\mathfrak{p}}\eta_{\mathfrak{p}}),(\kappa_{\hat{\mathfrak{q}},s} a_{\hat{\mathfrak{q}},\mathfrak{q}'}\xi_{\mathfrak{q}'}')\phi^{-\frac{1}{2}}(\kappa_{\hat{\mathfrak{p}},t} a_{\hat{\mathfrak{p}},\mathfrak{p}'}\eta_{\mathfrak{p}'}')\rangle\\
&&=\langle\kappa_{\hat{\mathfrak{q}}\lor\hat{\mathfrak{p}},s+t}(( a_{\hat{\mathfrak{q}},\mathfrak{q}}\xi_{\mathfrak{q}})\phi^{-\frac{1}{2}}( a_{\hat{\mathfrak{p}},\mathfrak{p}}\eta_{\mathfrak{p}})),\kappa_{\hat{\mathfrak{q}}\lor\hat{\mathfrak{p}},s+t}(( a_{\hat{\mathfrak{q}},\mathfrak{q}'}\xi_{\mathfrak{q}'}')\phi^{-\frac{1}{2}}( a_{\hat{\mathfrak{p}},\mathfrak{p}'}\eta_{\mathfrak{p}'}'))\rangle\\
&&=\langle( a_{\hat{\mathfrak{q}},\mathfrak{q}}\otimes^M a_{\hat{\mathfrak{p}},\mathfrak{p}})(\xi_{\mathfrak{q}}\phi^{-\frac{1}{2}}\eta_{\mathfrak{q}}),( a_{\hat{\mathfrak{q}},\mathfrak{q}'}\otimes^M a_{\hat{\mathfrak{p}},\mathfrak{p}'})(\xi_{\mathfrak{q}'}'\phi^{-\frac{1}{2}}\eta_{\mathfrak{q}'}')\rangle\\
&&=\langle a_{\hat{\mathfrak{q}}\lor\hat{\mathfrak{p}},\mathfrak{q}\lor\mathfrak{p}}(\xi_{\mathfrak{q}}\phi^{-\frac{1}{2}}\eta_{\mathfrak{q}}), a_{\hat{\mathfrak{q}}\lor\hat{\mathfrak{p}},\mathfrak{q}'\lor\mathfrak{p}'}(\xi_{\mathfrak{q}'}'\phi^{-\frac{1}{2}}\eta_{\mathfrak{q}'}')\rangle=\langle\kappa_{\mathfrak{q}\lor\mathfrak{p},s+t}(\xi_\mathfrak{q}\phi^{-\frac{1}{2}}\eta_{\mathfrak{p}}),\kappa_{\mathfrak{q'}\lor\mathfrak{p'},s+t}(\xi_\mathfrak{q'}'\phi^{-\frac{1}{2}}\eta_{\mathfrak{p'}}')\rangle.
\end{eqnarray*}
In particular, we conclude that $U_{s,t}^T$ is well-defined and can be extended to an isometry from $\mathcal{H}_s^T\otimes^M\mathcal{H}_t^T$ to $\mathcal{H}_{s+t}^T$, and also denote the isometry by $U_{s,t}^T$ again. The surjectivity and the two-sides linearity of $U_{s,t}^T$ are obvious.


To show (\ref{productsystemunitary}), it is enough to check it for $\phi$-bounded vectors with the form $\kappa_{\mathfrak{p},t}((x_1\otimes_{t_1}y_1\phi^{\frac{1}{2}})\phi^{-\frac{1}{2}}\cdots\phi^{-\frac{1}{2}}(x_m\otimes_{t_m}y_m\phi^{\frac{1}{2}}))$ for some $x_i,y_i\in
M$.

We conclude that $H^T$ is a product system of W$^*$-bimodules, and it is clear that $\Xi^T$ is a generating unital unit of $H^T$ by the definition of $\{U_{s,t}^T\}_{s,t\geq0}$.
\end{proof}
\begin{exam}\label{isocp}
Let $T=\{T_t\}_{t\geq0}$ be an algebraic CP$_0$-semigroup obtained by a semigroup $v$ of isometries in a von Neumann algebra $M$ in {\rm
Example \ref{cpexam0}}. For each $t\geq0$, we can identify $M\otimes_tL^2(M)$ with $L^2(M)$ by a bilinear unitary $U^t(x\otimes_ty\phi^\frac{1}{2})=xv_ty\phi^\frac{1}{2}$ for $x,y\in
M$. For $t\geq0$ and $\mathfrak{p}\in\mathfrak{P}_t$, the unitaries induce a bilinear unitary $U^\mathfrak{p}:\mathcal{H}^T(\mathfrak{p},t)\to
L^2(M)$ such that $U^\mathfrak{p} a_{\mathfrak{p},\mathfrak{q}}=U^\mathfrak{q}$ for all $\mathfrak{p}\succ\mathfrak{q}$. 

\end{exam}
\begin{exam}
We consider the CP$_0$-semigroup generated by a family of stochastic matrices. Let $M=\mathbb{C}\oplus\mathbb{C}$ be a von Neumann algebra regarded as a von Neumann subalgebra of $M_2(\mathbb{C})$. Then, we have $L^2(M)=\mathbb{C}\oplus\mathbb{C}$. Let $T=\{T_t\}_{t\geq0}$ be the CP$_0$-semigroup on $M$ associated with stochastic matrices $\left\{\begin{pmatrix}e^{-t}&1-e^{-t}\\0&1\end{pmatrix}\right\}_{t\geq0}$, that is, each $T_t$ is defined by $T_t(a\oplus
b)=(e^{-t}a+(1-e^{-t})b)\oplus
b$ for each $a,b\in\mathbb{C}$. By using the normalized canonical trace on $M_2(\mathbb{C})$, for $t\geq0$ and $\mathfrak{p}=(t_1,\cdots,t_n)\in\mathfrak{P}_t$, it turns out that $\mathcal{H}^T(\mathfrak{p},t)$ coincides with $\mathbb{C}^n\oplus\mathbb{C}^2$ on which $M$ acts as
\begin{eqnarray*}
&&(a\oplus
b)(x_1\oplus\cdots
x_n\oplus
y\oplus
z)=ax_1\oplus
bx_2\oplus\cdots\oplus
bx_n\oplus
by\oplus
bz\\
&&(x_1\oplus\cdots
x_n\oplus
y\oplus
z)(a\oplus
b)=ax_1\oplus\cdots\oplus
ax_n\oplus
ay\oplus
bz
\end{eqnarray*}
for $a,b,x_1,\cdots,x_n,y,z\in\mathbb{C}$. The W$^*$-bimodule $\mathcal{H}^T(\mathfrak{p},t)$ depends on only $T,\ t$ and the number $n$ of the partition $\mathfrak{p}$, and $\mathcal{H}_t^T$ is infinite dimensional for all $t>0$.
\end{exam}
\begin{exam}
We consider the W$^*$-$\mathcal{B}(\mathcal{H})$-bimodule $\mathcal{H}^T(\mathfrak{p},t)$ associated with the CCR heat flow in {\rm
Example \ref{heat}} for $t\geq0$ and $\mathfrak{p}\in\mathfrak{P}_t$.

Let $\mathcal{H}=L^2(\mathbb{R})$ and $M=\mathcal{B}(\mathcal{H})$.
Then, the standard space $L^2(\mathcal{B}(\mathcal{H}))$ of $M$ is isomorphic to 
$\mathcal{H}\otimes\mathcal{H}^*\cong\mathcal{C}_2(\mathcal{H})$. For $t\geq0,\ x,x'\in
M$ and $\xi\otimes\eta^*,\xi'\otimes{\eta'}^*\in\mathcal{H}\otimes\mathcal{H}^*$, the inner product on $M\otimes_tL^2(M)$ is given by $\langle
x\otimes(\xi\otimes\eta^*),x'\otimes(\xi'\otimes{\eta'}^*)\rangle=\langle\eta',\eta\rangle\int_{\mathbb{R}^2}\langle
xW_{\frac{{\bf
x}}{\sqrt{2}}}^*\xi,x'W_{\frac{{\bf
x}}{\sqrt{2}}}^*\xi'\rangle
d\mu_t({\bf
x}).$ Let $\mathfrak{p}=(t_1,\cdots,t_n)\in\mathfrak{P}_t$. Fix a faithful normal state $\phi$ on $M$ and suppose $\rho\in
\mathcal{C}_1(\mathcal{H})$ is associated with $\phi$ by $\phi(x)={\rm
tr}(\rho
x)$ for all $x\in
M$. In terms of $\mathcal{C}_2(\mathcal{H})$, the inner product on $\mathcal{H}^T(\mathfrak{p},t)$ is
\begin{eqnarray*}
&&\hspace{-20pt}\langle(x_1\otimes_{t_1}a_1\rho^\frac{1}{2})\phi^{-\frac{1}{2}}\cdots\phi^{-\frac{1}{2}}(x_n\otimes_{t_n}a_n\rho^\frac{1}{2}),(y_1\otimes_{t_1}b_1\rho^\frac{1}{2})\phi^{-\frac{1}{2}}\cdots\phi^{-\frac{1}{2}}(y_n\otimes_{t_n}b_n\rho^\frac{1}{2})\rangle\\
&&\hspace{-20pt}=\int_{\mathbb{R}^{2n}}\langle
x_1W_{\frac{{\bf
x}_1}{\sqrt{2}}}^*a_1\cdots
x_nW_{\frac{{\bf
x}_n}{\sqrt{2}}}^*a_n\rho^\frac{1}{2},y_1W_{\frac{{\bf
x}_1}{\sqrt{2}}}^*b_1\cdots
y_nW_{\frac{{\bf
x}_n}{\sqrt{2}}}^*b_n\rho^\frac{1}{2}\rangle
d\mu_{t_1}({\bf
x}_1)\cdots
d\mu_{t_n}({\bf
x}_n)
\end{eqnarray*}
for each $x_1,\cdots,x_n,y_1,\cdots,y_n,a_1,\cdots,a_n,b_1,\dots,b_n\in
M$. The properties {\rm(\ref{Weyl})} and $\mu_s*\mu_t=\mu_{s+t}$ ensure the fact that $ a_{\mathfrak{p},\mathfrak{q}}$ defined by {\rm(\ref{isometries})} is isometry for $\mathfrak{p},\mathfrak{q}\in\mathfrak{P}_t$. 
\end{exam}
In Section \ref{Heat semigroups on manifolds and product systems}, we also will discuss the case of classical heat semigroups, that is, the construction of product system associated with the heat semigroup on a compact Riemannian manifold and its dilation in more detail.\\

We denote the inductive limit of the pair $(H^T,\Xi^T)$ associated with an algebraic CP$_0$-semigroup $T$ by $\mathfrak{H}^T$, and the unitary: $\mathfrak{H}^T\otimes^M\mathcal{H}_t^T\to\mathfrak{H}^T$ in Theorem \ref{U_tgeneral} by $U_t^T$. Also, $\{b_{t,s}^T\}_{s\leq
t}$ is the family of isometries giving the inductive system for $H^T$ as (\ref{inductivebeta}) and $\pi^T$ is the faithful representation of $M$ on $\mathfrak{H}^T$ defined as (\ref{faithfulrepM}).
\begin{prop}\label{continuityp}
If $T$ is a CP$_0$-semigroup, then the unit $\Xi^T$ is continuous.
\end{prop}
\begin{proof}
Suppose that $t<\min\{s_1,\cdots,s_m\}$ and $\xi=\kappa_s\kappa_{\mathfrak{q},s}((x_1\otimes_{s_1}y_1\phi^{\frac{1}{2}})\phi^{-\frac{1}{2}}\cdots\phi^{-\frac{1}{2}}(x_m\otimes_{s_m}y_m\phi^{\frac{1}{2}}))$ for some $s\geq0,\ \mathfrak{q}=(s_1,\cdots,s_m)\in\mathfrak{P}_s,\ x_1,\cdots,x_m,y_1,\cdots,y_m\in
M$. If we put $\mathfrak{p}'=(t,s_1-t,t,s_2-t,t,\cdots,s_m-t,t)$, then we have $\mathfrak{p}'\succ(t)\lor\mathfrak{p},\mathfrak{p}\lor(t)$. Now, we have
\begin{eqnarray*}
\xi&=&\kappa_{s+t}b_{s+t,s}^T\kappa_{\mathfrak{q},s}((x_1\otimes_{s_1}y_1\phi^{\frac{1}{2}})\phi^{-\frac{1}{2}}\cdots\phi^{-\frac{1}{2}}(x_m\otimes_{s_m}y_m\phi^{\frac{1}{2}}))\\
&=&\kappa_{s+t}\kappa_{(t)\lor\mathfrak{p},s+t}((1_M\otimes_t\phi^\frac{1}{2})\phi^{-\frac{1}{2}}(x_1\otimes_{s_1}y_1\phi^{\frac{1}{2}})\phi^{-\frac{1}{2}}\cdots\phi^{-\frac{1}{2}}(x_m\otimes_{s_m}y_m\phi^{\frac{1}{2}}))\\
&=&\kappa_{s+t}\kappa_{\mathfrak{p}',s+t} a_{\mathfrak{p}',(t)\lor\mathfrak{p}}((1_M\otimes_t\phi^\frac{1}{2})\phi^{-\frac{1}{2}}(x_1\otimes_{s_1}y_1\phi^{\frac{1}{2}})\phi^{-\frac{1}{2}}\cdots\phi^{-\frac{1}{2}}(x_m\otimes_{s_m}y_m\phi^{\frac{1}{2}}))\\
&=&\kappa_{s+t}\kappa_{\mathfrak{p}',s+t}((1_M\otimes_t\phi^\frac{1}{2})\phi^{-\frac{1}{2}}((x_1\otimes_{s_1-t}\phi^{\frac{1}{2}})\phi^{-\frac{1}{2}}(1_M\otimes_{t}y_1\phi^{\frac{1}{2}}))\phi^{-\frac{1}{2}}\\
&&\hspace{170pt}\cdots\phi^{-\frac{1}{2}}((x_m\otimes_{s_m-t}\phi^{\frac{1}{2}})\phi^{-\frac{1}{2}}(1_M\otimes_{t}y_m\phi^{\frac{1}{2}}))).
\end{eqnarray*}
On the other hand, we have
\begin{eqnarray*}
&&U_t(\xi\phi^{-\frac{1}{2}}\xi^T(t))=\kappa_{s+t}\kappa_{\mathfrak{q}\lor(t),s+t}((x_1\otimes_{s_1}y_1\phi^{\frac{1}{2}})\phi^{-\frac{1}{2}}\cdots\phi^{-\frac{1}{2}}(x_m\otimes_{s_m}y_m\phi^{\frac{1}{2}})\phi^{-\frac{1}{2}}(1_M\otimes_t\phi^\frac{1}{2}))\\
&&=\kappa_{s+t}\kappa_{\mathfrak{p}',s+t} a_{\mathfrak{p}',\mathfrak{p}\lor(t)}((x_1\otimes_{s_1}y_1\phi^{\frac{1}{2}})\phi^{-\frac{1}{2}}\cdots\phi^{-\frac{1}{2}}(x_m\otimes_{s_m}y_m\phi^{\frac{1}{2}})\phi^{-\frac{1}{2}}(1_M\otimes_t\phi^\frac{1}{2}))\\
&&=\kappa_{s+t}\kappa_{\mathfrak{p}',s+t}(((x_1\otimes_t\phi^\frac{1}{2})\phi^{-\frac{1}{2}}(1_M\otimes_{s_1-t}y_1\phi^{\frac{1}{2}}))\phi^{-\frac{1}{2}}\\
&&\hspace{120pt}\cdots\phi^{-\frac{1}{2}}((x_m\otimes_t\phi^{\frac{1}{2}})\phi^{-\frac{1}{2}}(1_M\otimes_{s_m-t}y_m\phi^{\frac{1}{2}}))\phi^{-\frac{1}{2}}(1_M\otimes_{t}\phi^{\frac{1}{2}})).
\end{eqnarray*}
By calculations of inner products and \cite[Lemma A.2]{skei16}, when $t$ tend to 0, we conclude that $U_t(\xi\phi^{-\frac{1}{2}}\xi^T(t))$ converges to $\xi$.
\end{proof}

The dilation of the pair $(H^T,\Xi^T)$ associated with a CP$_0$-semigroup gives a dilation of $T$ as follows:
\begin{theo}\label{dilation2}
Let $T$ be a CP$_0$-semigroup on a von Neumann algebra $M$ and $\theta$ the dilation of the pair $(H^T,\Xi^T)$. The triple $({\rm
End}(\mathfrak{H}_M^T),\pi^T(1_M),\theta)$ is a dilation of $T$. Moreover, if we denote by $N$ the von Neumann algebra generated by $\bigcup_{t\geq0}\theta_t(\pi^T(M))$, the triple $(N,\pi^T(1_M),\theta|_N)$ is the minimal dilation of $T$.
\end{theo}
\begin{proof}
We shall show that $\pi^T(T_t(x))=p\theta_t(\pi^T(x))p$ for all $t\geq0$ and $x\in
M$. For all $y,z\in
M$, we have
\begin{eqnarray*}
&&\langle\phi^\frac{1}{2}y,\kappa_0^*\theta_t(\pi^T(x))\kappa_0\phi^\frac{1}{2}z)\rangle=\langle
(\kappa_0\phi^\frac{1}{2})\phi^{-\frac{1}{2}}\kappa_t^*\kappa_0\phi^\frac{1}{2}y,(\pi^T(x)\otimes^M{\rm
id}_{\mathcal{H}_t})(\kappa_0\phi^\frac{1}{2})\phi^{-\frac{1}{2}}\kappa_t^*\kappa_0\phi^\frac{1}{2}z\rangle\\
&&=\langle
(\kappa_0\phi^\frac{1}{2})\phi^{-\frac{1}{2}}\kappa_t^*\kappa_0\phi^\frac{1}{2}y,\kappa_0(x\phi^\frac{1}{2})\phi^{-\frac{1}{2}}\kappa_t^*\kappa_0\phi^\frac{1}{2}z\rangle=\langle
(\kappa_0\phi^\frac{1}{2})\phi^{-\frac{1}{2}}b_{t,0}^T\phi^\frac{1}{2}y,\kappa_0(x\phi^\frac{1}{2})\phi^{-\frac{1}{2}}b_{t,0}^T\phi^\frac{1}{2}z\rangle\\
&&=\langle
(\kappa_0\phi^\frac{1}{2})\phi^{-\frac{1}{2}}\kappa_{(t),t}(1_M\otimes_t\phi^\frac{1}{2}y),\kappa_0(x\phi^\frac{1}{2})\phi^{-\frac{1}{2}}\kappa_{(t),t}(1_M\otimes_t\phi^\frac{1}{2}z)\rangle=\langle
\phi^\frac{1}{2}y,T_t(x)\phi^\frac{1}{2}z\rangle.
\end{eqnarray*}
Thus we have $T_t(x)=\kappa_0^*\theta_t(\pi(x))\kappa_0$.

For $t\geq0,\ \mathfrak{p}=(t_1,\cdots,t_n)\in\mathfrak{P}_t,\ x_1,\cdots,x_n,y\in
M$, we can check that
\begin{eqnarray*}
&&\theta_{t}(\pi^T(x_1))\theta_{t-t_1}(\pi^T(x_2))\cdots\theta_{t_{n-1}+t_n}(\pi^T(x_{n-1}))\theta_{t_n}(\pi^T(x_n))\kappa_0\phi^{\frac{1}{2}}y\\
&&=\kappa_t\kappa_{\mathfrak{p},t}((x_1\otimes_{t_1}\phi^{\frac{1}{2}})\phi^{-\frac{1}{2}}(x_2\otimes_{t_2}\phi^{\frac{1}{2}})\phi^{-\frac{1}{2}}\cdots\phi^{-\frac{1}{2}}(x_{n-1}\otimes_{t_{n-1}}\phi^{\frac{1}{2}})\phi^{-\frac{1}{2}}(x_n\otimes_{t_n}\phi^{\frac{1}{2}}y)).
\end{eqnarray*}
Hence, we have $\overline{{\rm
span}}(N\pi(1_M)\mathfrak{H}^T)\supset\overline{{\rm
span}}(N\kappa_0L^2(M))=\mathfrak{H}^T$. Since the central support $c(\pi^T(1_M))$ of $\pi^T(1_M)$ in $N$ is the projection onto $\overline{{\rm
span}}(N\pi(1_M)\mathfrak{H}^T)$, we have $c(\pi^T(1_M))=1_N$. We conclude that the triplet $(N,\pi(1_M),\theta)$ is the minimal dilation of $T$.
\end{proof}
\begin{rema}
Let $\theta$ be the dilation of a CP$_0$-semigroup $T$ defined by (\ref{maxdilation}). When $s\leq
t$ and $\xi_s\in\mathcal{H}_s^T$, we have
\begin{eqnarray*}
&&\theta_t(a)\kappa_s\xi_s=U_t(a\otimes^M{\rm
id}_{\mathcal{H}_t^T})U_t^*\kappa_s\xi_s=U_t(a(\kappa_0\phi^{\frac{1}{2}})\phi^{-\frac{1}{2}}b_{t,s}\kappa_s\xi_s)
\end{eqnarray*}
for all $a\in{\rm
End}(\mathfrak{H}_M^T)$, and hence the operator $\theta_t(a)\kappa_s\xi_s$ is depend on only the image $a(\kappa_0\phi^{\frac{1}{2}})$.
\end{rema}
The following theorem asserts that the correspondence between algebraic CP$_0$-semigroups and generating unital units is one-to-one up to unit preserving isomorphism.
\begin{theo}\label{correspondenceCP}
Let $T$ be an algebraic CP$_0$-semigroup on $M$ and $\Xi$ a generating unital unit of a $\phi$-spatial product system $H$ of W$^*$-$M$-bimodule. Then, we have $T^{\Xi^T}=T$ and there exists an isomorphism from $H^{T^\Xi}$ onto $H$, which preserves the units $\Xi^{T^\Xi}$ and $\Xi$. 
\end{theo}
\begin{proof}
It is clear that $T^{\Xi^T}=T$. For $t\geq0$ and a partition $\mathfrak{p}=(t_1,\cdots,t_n)\in\mathfrak{P}_t$, we define a map $u_t:\mathcal{H}_t^{\Xi^T}\to\mathcal{H}_t$ by
\begin{eqnarray*}
&&u_t(\kappa_{\mathfrak{p},t}((x_1\otimes_{t_1}\phi^\frac{1}{2})\phi^{-\frac{1}{2}}\cdots\phi^{-\frac{1}{2}}(x_{n-1}\otimes_{t_{n-1}}\phi^\frac{1}{2})\phi^{-\frac{1}{2}}(x_n\otimes_{t_n}\phi^\frac{1}{2}y)))\\
&&=U(\mathfrak{p})(x_1\xi(t_1)\phi^{-\frac{1}{2}}\cdots
\phi^{-\frac{1}{2}}x_{n-1}\xi(t_{n-1})\phi^{-\frac{1}{2}}x_n\xi(t_n)y)
\end{eqnarray*}
for each $x_1,\cdots,x_n,y\in
M$, where $\kappa_{\mathfrak{p},t}:\mathcal{H}^{T^\Xi}(\mathfrak{p},t)\to\mathcal{H}_t^{T^\Xi}$ is the canonical embedding. We can check that $u_t$ is an isometry. Since $\Xi$ is generating, $u_t$ can be extended as unitary from $\mathcal{H}_t^{T^\Xi}$ onto $\mathcal{H}_t$.

We must show that $U_{s,t}((u_s\xi_s)\phi^{-\frac{1}{2}}(u_t\eta_t))=u_{s+t}U_{s,t}^{T^\Xi}(\xi_s\phi^{-\frac{1}{2}}\eta_t)$ for all $\xi_s\in\mathcal{D}(\mathcal{H}_s^{T^\Xi};\phi)$ and all $\eta_t\in\mathcal{H}_t^{T^\Xi}$. It enough to show it for 
\begin{eqnarray*}
&&\xi_s=\kappa_{\mathfrak{q},s}((x_1\otimes_{s_1}\phi^\frac{1}{2})\phi^{-\frac{1}{2}}\cdots\phi^{-\frac{1}{2}}(x_{m-1}\otimes_{s_{m-1}}\phi^\frac{1}{2})\phi^{-\frac{1}{2}}(x_m\otimes_{s_m}\phi^\frac{1}{2})),\\
&&\eta_t=\kappa_{\mathfrak{p},t}((z_1\otimes_{t_1}\phi^\frac{1}{2})\phi^{-\frac{1}{2}}\cdots\phi^{-\frac{1}{2}}(z_{n-1}\otimes_{t_{n-1}}\phi^\frac{1}{2})\phi^{-\frac{1}{2}}(z_n\otimes_{t_n}\phi^\frac{1}{2}w)),
\end{eqnarray*}
where $\mathfrak{q}=(s_1,\cdots,s_m)\in\mathfrak{P}_s,\ \mathfrak{p}=(t_1,\cdots,t_n)\in\mathfrak{P}_t$ and $x_1,\cdots,x_m,z_1,\cdots,z_n,$ $w\in
M$. We put $\zeta_1=x_1\xi(s_1)\phi^{-\frac{1}{2}}\cdots
\phi^{-\frac{1}{2}}x_{m}\xi(s_{m}),\ \zeta_2=z_1\xi(t_1)\phi^{-\frac{1}{2}}\cdots
\phi^{-\frac{1}{2}}z_{n-1}\xi(t_{n-1})\phi^{-\frac{1}{2}}z_n\xi(t_n)w.$ By the associativity of $\{U_{s,t}\}_{s,t\geq0}$, we have
\begin{eqnarray*}
&&u_{s+t}U_{s,t}^{T^\Xi}(\xi_s\phi^{-\frac{1}{2}}\eta_t)\\
&&=U_{s_1,s_1'+t}({\rm
id}_{s_1}\otimes
U_{s_2,s_2'+t})({\rm
id}_{s_1,s_2}\otimes
U_{s_3,s_3'+t})\cdots({\rm
id}_{s_1,\cdots,s_{m-1}}\otimes
U_{s_m,t})(\zeta_1\phi^{-\frac{1}{2}}U(\mathfrak{p})\zeta_2)\\
&&=U_{s,t}(U_{s_1,s_1'}\otimes{\rm
id}_t)({\rm
id}_{s_1}\otimes
U_{s_2,s_2'}\otimes{\rm
id}_{t})({\rm
id}_{s_1,s_2}\otimes
U_{s_3,s_3'}\otimes{\rm
id}_{t})\\
&&\hspace{160pt}\cdots({\rm
id}_{s_1,\cdots,s_{n-2}}\otimes
U_{s_{n-1},s_{n-1}'}\otimes{\rm
id}_{t})(\zeta_1\phi^{-\frac{1}{2}}U(\mathfrak{p})(\zeta_2))\\
&&=U_{s_1,s_1'+t}({\rm
id}_{s_1}\otimes
U_{s_2,s_2'+t})\cdots({\rm
id}_{s_1,\cdots,s_{m-1}}\otimes
U_{s_m,t})({\rm
id}_{s_1,\cdots,s_m}\otimes
U_{t_1,t_1'})\nonumber\\
&&\hspace{170pt}\cdots({\rm
id}_{s_1,\cdots,s_m}\otimes{\rm
id}_{t_1,\cdots,t_{n-2}}\otimes
U_{t_{n-1},t_{n-1}'})(\zeta_1\phi^{-\frac{1}{2}}\zeta_2)\\
&&=U_{s,t}((u_s\xi_s)\phi^{-\frac{1}{2}}(u_t\eta_t)).
\end{eqnarray*}
We conclude that $\{u_t\}_{t\geq0}$ gives an isomorphism.
\end{proof}
As a corollary of Theorem \ref{correspondenceCP}, it will turn out that every weakly continuous generating unit is continuous. For this, we shall show that an isomorphism between product systems of W$^*$-bimodules induces a unitary between their inductive limits as follows: let $H=\{\mathcal{H}_t\}_{t\geq0}$ and $K=\{\mathcal{K}_t\}_{t\geq0}$ be product systems of W$^*$-$M$-bimodules with unital units $\Xi$ and $\Lambda$, respectively. Suppose a family $=\{u_t\}_{t\geq0}$ is an isomorphism from $H$ onto $K$. Then, we can define the canonical right $M$-linear unitary $u$ from the inductive limit $\mathfrak{H}$ of $(H,\Xi)$ onto the one $\mathfrak{K}$ of $(K,\Lambda)$ by $u(\kappa_t^H(\xi_t))=\kappa_t^Ku_t(\xi_t)$ for each $t\geq0$ and $\xi_t\in\mathcal{H}_t$, where $\kappa_t^H:\mathcal{H}_t\to\mathfrak{H}$ and $\kappa_t^K:\mathcal{K}_t\to\mathfrak{K}$ are the canonical embedding.

\begin{coro}\label{strongcontinuity}
If $\Xi$ is a weakly continuous generating unital unit of a product system $H$ of W$^*$-bimodules, then $\Xi$ is continuous.
\end{coro}
\begin{proof}
Let $U_t^{T^\Xi}$ be the unitary giving the right W$^*$-$M$-module isomorphism $\mathfrak{H}^{T^\Xi}\otimes^M\mathcal{H}_t^{T^\Xi}\to\mathfrak{H}^{T^\Xi}$ for each $t\geq0$. By Proposition \ref{continuityp}, the unit $\Xi^{T^\Xi}=\{\xi^{T^\Xi}(t)\}_{t\geq0}$ satisfies
\begin{eqnarray}\label{continuity3}
U_t^{T^\Xi}(\xi\phi^{-\frac{1}{2}}\xi^{T^\Xi}(t))\to\xi\hspace{20pt}(t\to+0)
\end{eqnarray}
for all $\xi\in\mathcal{D}(\mathcal{H}^{T^\Xi};\phi)$. Suppose $\mathfrak{H}$ and $\mathfrak{H}^{T^\Xi}$ are the inductive limit of $(H,\Xi)$ and $(H^{T^\Xi},\Xi^{T^\Xi})$, respectively, and $\kappa_t:\mathcal{H}_t\to\mathfrak{H}$ and $\kappa_t^{T^\Xi}:\mathcal{H}_t^{T^\Xi}\to\mathfrak{H}^{T^\Xi}$ are the canonical embedding for each $t\geq0$. Let $u$ be the unitary from $\mathfrak{H}^{T^\Xi}$ onto $\mathfrak{H}$ induced from the isomorphism $\{u_t\}_{t\geq0}$ which is obtained by Theorem \ref{correspondenceCP}. For each $\xi_s\in\mathcal{H}_s^{T^\Xi}$, we have $uU_t^{T^\Xi}(\kappa_s^{T^\Xi}\xi_s\phi^{-\frac{1}{2}}\xi^{T^\Xi}(t))=U_t((u\kappa_s^{T^\Xi}\xi_s)\phi^{-\frac{1}{2}}u_t\xi^{T^\Xi}(t))$ by (\ref{isomorphism}) and (\ref{isomorphism}). Thus, by (\ref{continuity3}) and the fact that $\{u_t\}_{t\geq0}$ is unit preserving, we have $\|U_t(\xi\phi^{-\frac{1}{2}}\xi(t))-\xi\|=\|U_t((uu^*\xi)\phi^{-\frac{1}{2}}(u_tu_t^*\xi(t)))-\xi\|=\|U_t^{T^\Xi}((u^*\xi)\phi^{-\frac{1}{2}}\xi^{T^\Xi}(t))-u^*\xi\|$, and hence $\Xi$ is continuous. 
\end{proof}

\begin{rema}
Let $T$ be a CP$_0$-semigroup on a von Neumann algebra acting on a separable Hilbert space $\mathcal{H}$. The product system $H^T=\{\mathcal{H}_t^T\}_{t\geq0}$ of W$^*$-bimodules associated with $T$ gives a relation between Bhat-Skeide's {\rm(\cite{bhat-skei00})} and Muhly-Solel's {\rm(\cite{muhl-sole02})} constructions of the minimal dilation of $T$ as follows: a common point of the two methods is to establish a product system of von Neumann bimodules which has an information of $T$ by taking the inductive limits with respect to refinements of partitions. Let $\{E_t^T\}_{t\geq0}$ and $\{E^T(t)\}_{t\geq0}$ be the product systems of von Neumann bimodules associated with $T$ appearing in Bhat-Skeide's and Muhly-Solel's constructions, respectively. Note that each $E_t^T$ is a von Neumann $M$-bimodule and each $E^T(t)$ is a von Neumann $M'$-bimodule. By the correspondence between W$^*$-bimodules and von Neumann bimodules {\rm(}see {\rm\cite[Section 2]{skei03})}, we have a correspondence.
\begin{equation}\label{relationbsms}
E_t\longleftrightarrow\mathcal{H}_t^T,\ E(t)\longleftrightarrow\mathcal{H}^*\otimes^M\mathcal{H}_t^T\otimes^M\mathcal{H}.
\end{equation}
for each $t\geq0$. This is an extension to the continuous case of the relation in the discrete case given by {\rm\cite{sawa17}} and the proof of the correspondence {\rm(\ref{relationbsms})} is essentially the same.
\end{rema}

\section{Heat semigroups on manifolds and product systems}\label{Heat semigroups on manifolds and product systems}
In this section, we will consider the product system associated a heat semigroup $T$ and a dilation of $T$.

We shall recall the concept of heat semigroup on a compact Riemannian manifold. We refer the reader to \cite{grig} for their general theory. Let $\mathcal{M}$ be a compact Riemannian manifold with the normalized Riemannian measure $\mu$ associated with $\mathcal{M}$. We can define the self-adjoint positive (unbounded) operator $\Delta$ on $L^2(\mathcal{M})$ like the Laplacian on the Euclid space. The operator $\Delta$ is called the Laplacian (or Dirichlet Laplacian) on $\mathcal{M}$. The semigroup $T=\{T_t\}_{t\geq0}$ of bounded operators on $L^2(\mathcal{M})$ defined by $T_t=e^{-t\Delta}$ for each $t\geq0$, is called the heat semigroup on $\mathcal{M}$. 

For $t>0$, it is known that there exists a measurable function $p_t$ on $\mathcal{M}\times\mathcal{M}$ such that
\begin{equation}\label{heatkernel}
(T_tf)(x)=\int_\mathcal{M}p_t(x,y)f(y)d\mu(y)
\end{equation}
for each $f\in
L^2(\mathcal{M})$. The family $\{p_t\}_{t>0}$ called the heat kernel on $\mathcal{M}$ has the following properties:
\begin{enumerate}[(1)]
\item
$p_t(x,y)=p_t(y,z)\geq0\ (t>0,\ x,y\in\mathcal{M}).$
\item
$\int_\mathcal{M}p_t(x,y)d\mu(y)=1\ (t>0,\ x\in\mathcal{M}).$
\item
$p_{s+t}(x,z)=\int_\mathcal{M}p_t(x,y)p_t(y,z)d\mu(z)\ (s,t>0,\ x,z\in\mathcal{M}).$
\end{enumerate}
The equation (\ref{heatkernel}) enables as to extend the heat semigroup $T$ to a semigroup on $L^p(\mathcal{M})$ for each $1\leq
p\leq\infty$. We have the following continuity with respect to parameters:
\begin{equation}\label{heatcontinuity}
T_tf\overset{L^p}{\to}f\ (t\to+0)
\end{equation}
for each $f\in
L^p(\mathcal{M})$ and $p=1,2$.

The heat semigroup $T$ on $\mathcal{M}$ is a CP$_0$-semigroup on the commutative von Neumann algebra $M=L^\infty(\mathcal{M})$. A Noncommutative Laplacian and the associated CP$_0$-semigroup on a type I factor are discussed in \cite[Chapter 7]{arve03}. They are noncommutative analogies of the Laplacian on and heat semigroup on a manifold.

Now, we shall compute the product system $H^T$ of W$^*$-bimodules associated with the heat semigroup $T$ on $M=L^\infty(\mathcal{M})$. For this, we introduce the following notations.
\begin{defi}\label{M_p,mu_p}
For $t>0$ and $\mathfrak{p}=(t_1,\cdots,t_n)\in\mathfrak{P}_t$, we define a probability measure $\mu_{\mathfrak{p}}$ on $\mathcal{M}_\mathfrak{p}=\mathcal{M}^{n+1}$ by
\[
\mu_{\mathfrak{p}}=p_{t_1}(x_1,x_2)p_{t_2}(x_2,x_3)\cdots
p_{t_{n-1}}(x_{n-1},x_n)p_{t_n}(x_n,x_{n+1})\mu^{n+1}.
\]
For convenience, we define as $\mathcal{M}_{()}=\mathcal{M}$ and $\mu_{()}=\mu$ for the empty partition $()$.
\end{defi}
\begin{defi}
Let $s,t>0,\ \mathfrak{p}\in\mathfrak{P}_s,\ \mathfrak{q}\in\mathfrak{P}_t$ with $\#\mathfrak{p}=m,\ \#\mathfrak{q}=n$, and $f_\mathfrak{p},g_\mathfrak{q}$ and $h$ be functions on $\mathcal{M}_\mathfrak{p},\mathcal{M}_\mathfrak{q}$ and $\mathcal{M}$, respectively. We define functions $f_\mathfrak{p}\Box
g_\mathfrak{q},f_\mathfrak{p}\Box
h$ and $h\Box
g_\mathfrak{q}$ on $\mathcal{M}_{\mathfrak{p}\lor\mathfrak{q}},\mathcal{M}_\mathfrak{p}$ and $\mathcal{M}_\mathfrak{q}$, respectively, by
\begin{eqnarray*}
&&(f_\mathfrak{p}\Box
g_\mathfrak{q})(x_1,\cdots,x_m,y_1,\cdots,y_n,z)=f_\mathfrak{p}(x_1,\cdots,x_m,y_1)g_\mathfrak{q}(y_1,\cdots,y_n,z),\\
&&(f_\mathfrak{p}\Box
h)(x_1,\cdots,x_m,x_{m+1})=f_\mathfrak{p}(x_1,\cdots,x_m,x_{m+1})h(x_{m+1}),\\
&&(h\Box
g_\mathfrak{q})(y_1,\cdots,y_n,y_{n+1})=h(y_1)g_\mathfrak{q}(y_1,\cdots,y_n,y_{n+1}).
\end{eqnarray*}
for $x_i,y_j,z\in\mathcal{M}$.
\end{defi}
For $t>0$ and $\mathfrak{p}\in\mathfrak{P}_t$, the Hilbert space $L^2(\mathcal{M}_\mathfrak{p},\mu_{\mathfrak{p}})$ has a canonical W$^*$-$M$-bimodule structure given by $gf=g\Box
f$ and $f\Box
g$ for each $f\in
L^2(\mathcal{M}_\mathfrak{p},\mu_{\mathfrak{p}})$ and $g\in
M$. Then, we can obtain the following identification as W$^*$-$M$-bimodules.

\begin{prop}\label{H^T(p,t)L^2}
For $t>0$ and $\mathfrak{p}\in\mathfrak{P}_t$, we have an isomorphism $H^T(\mathfrak{p},t)\cong
L^2(\mathcal{M}_\mathfrak{p},\mu_{\mathfrak{p}})$ as W$^*$-$M$-bimodules.
\end{prop}
\begin{proof}
Let $\tau$ be the canonical faithful normal trace on $M=L^\infty(\mathcal{M})$ given by integrals on $\mathcal{M}$. Suppose $\mathfrak{p}=(t_1,\cdots,t_n)$. We define a $M$-bilinear map $u_{\mathfrak{p},t}:H^T(\mathfrak{p},t)\to
L^2(\mathcal{M}_\mathfrak{p},\mu_{\mathfrak{p}})$ by
\[
u_{\mathfrak{p},t}((f_1\otimes_{t_1}g_1\tau^\frac{1}{2})\tau^{-\frac{1}{2}}\cdots\tau^{-\frac{1}{2}}(f_n\otimes_{t_n}g_n))=f_1(x_1)g_1(x_2)f_2(x_2)\cdots
g_{n-1}(x_n)f_n(x_n)g_n(x_{n+1})
\]
for each $f_1,g_1,\cdots,f_n,g_n\in
M$, where $f_i(x_i)$ denotes the function $f_i$ on $\mathcal{M}$ with variables $x_i$ and $g_i(x_{i+1})$ is similar. By the formula in Proposition \ref{formula}, we have
\begin{eqnarray*}
&&\langle(f_1\otimes_{t_1}g_1\tau^\frac{1}{2})\tau^{-\frac{1}{2}}\cdots\tau^{-\frac{1}{2}}(f_n\otimes_{t_n}g_n),(f_1'\otimes_{t_1}g_1'\tau^\frac{1}{2})\tau^{-\frac{1}{2}}\cdots\tau^{-\frac{1}{2}}(f_n'\otimes_{t_n}g_n')\rangle\\
&&=\int_{\mathcal{M}_\mathfrak{p}}\overline{f_1(x_1)g_1(x_2)f_2(x_2)\cdots
g_{n-1}(x_n)f_n(x_n)g_n(x_{n+1})}\\
&&\hspace{120pt}f_1'(x_1)g_1'(x_2)f_2'(x_2)\cdots
g_{n-1}'(x_n)f_n'(x_n)g_n'(x_{n+1})d\mu_{\mathfrak{p}},
\end{eqnarray*}
and hence $u_{\mathfrak{p},t}$ is an isometry. 

We shall check that $u_{\mathfrak{p},t}$ is surjective. For an arbitrary $\varepsilon>0$ and $f\in
L^2(\mathcal{M}_\mathfrak{p},\mu_{\mathfrak{p}})$, there exists $g\in
C(\mathcal{M}^{n+1})$ such that $\|f-g\|_{L^2(\mathcal{M}_\mathfrak{p},\mu_{\mathfrak{p}})}<\varepsilon$. Since the space
\[
{\rm
span}\{f_1(x_1)\cdots
f_{n+1}(x_{n+1})\in
C(\mathcal{M}^{n+1})\mid
f_i\in
C(\mathcal{M})\}
\]
is dense in $C(\mathcal{M}^{n+1})$ with respect to the uniform convergence topology, there exist $N\in\mathbb{N}$ and functions $f_{i,j}\in
C(\mathcal{M})$ for each $i=1,\cdots,n+1$ and $j=1,\cdots,N$ such that $\|g-\sum_{j=1}^Nf_{1,j}(x_1)\cdots
f_{n+1,j}(x_{n+1})\|_{\infty}<\varepsilon$. Now, equations
\[
f_{1,j}(x_1)\cdots
f_{n+1,j}(x_{n+1})=u_{\mathfrak{p},t}((f_{1,j}\otimes_{t_1}1_{M}\tau^\frac{1}{2})\tau^{-\frac{1}{2}}\cdots\tau^{-\frac{1}{2}}(f_{n-1,j}\otimes_{t_{n-1}}1_{M}\tau^\frac{1}{2})\tau^{-\frac{1}{2}}(f_{n,j}\otimes_{t_n}f_{n+1,j}\tau^{\frac{1}{2}}))
\]
imply that the image $u_{\mathfrak{p},t}(D_{\mathfrak{p},t})$ of
\[
D_{\mathfrak{p},t}={\rm
span}\{(f_1\otimes_{t_1}g_1\tau^\frac{1}{2})\tau^{-\frac{1}{2}}\cdots\tau^{-\frac{1}{2}}(f_n\otimes_{t_n}g_n)\mid
f_1,g_1,\cdots,f_n,g_n\in
M\}
\]
by $u_{\mathfrak{p},t}$ is dense in $L^2(\mathcal{M}_\mathfrak{p},\mu_{\mathfrak{p}})$, and hence $u_{\mathfrak{p},t}$ is unitary.
\end{proof}
\begin{rema}
There is a connection between heat kernels and Brownian motions on Riemannian manifolds. Let $T$ be a heat semigroup on a compact Riemannian manifold $\mathcal{M}$ and $\{p_t\}_{t>0}$ the heat kernel associated with $T$. For $x\in\mathcal{M}$, there exist a probability space $(\Omega_x,\mathbb{P}_x)$ and an $\mathcal{M}$-valued stochastic process $\{X_t\}_{t\geq0}$ such that
\[
\mathbb{P}_x(X_t\in
A)=\int_Ap_t(x,y)d\mu(y)=(T_t1_A)(x)
\]
for all Borel set $A\subset\mathcal{M}$, where $1_A$ is the characteristic function on $A$. Also, the family $\{\mu_\mathfrak{p}\mid\mathfrak{p}\in\bigcup_{t>0}\mathfrak{P}_t\}$ of probability measures describes joint distributions for $\{X_t\}_{t\geq0}$ as follows: for a partition $\mathfrak{p}=(t_1,\cdots,t_n)\in\mathfrak{P}_t$ and Borel sets $A_1,\cdots,A_n$, if we denote $\check{\mathfrak{p}}=(t_2,t_3,\cdots,t_n)$, then we have
\begin{eqnarray*}
&&\mathbb{P}_x(X_{t_1}\in
A_1,X_{t_1+t_2}\in
A_2,\cdots,X_{t_1+\cdots,t_{n-1}}\in
A_{n-1},X_{t}\in
A_n)\\
&&=\int_{\mathcal{M}_{\check{\mathfrak{p}}}}p_{t_1}(x,x_2)1_{A_1}(x_2)\cdots1_{A_{n-1}}(x_n)1_{A_n}(x_{n+1})d\mu_{\check{\mathfrak{p}}}.
\end{eqnarray*}
\end{rema}
Now, we shall reconstruct the product system associated with the heat semigroup $T$ and a dilation of $T$ under the identification in Proposition \ref{H^T(p,t)L^2}.

Let $\mathfrak{q}\succ\mathfrak{p}$ with $\mathfrak{p}=(t_1,\cdots,t_n)\in\mathfrak{P}_t$ and $\mathfrak{q}=\mathfrak{p}(t_1)\lor\cdots\lor\mathfrak{p}(t_n)$ with $\mathfrak{p}(t_i)=(t_{i,1},\cdots,t_{i,k(i)})\in\mathfrak{P}_{t_i}$. The isometry $a_{\mathfrak{q},\mathfrak{p}}:L^2(\mathcal{M}_\mathfrak{p},\mu_\mathfrak{p})\to
L^2(\mathcal{M}_\mathfrak{q},\mu_\mathfrak{q})$ is given by
\begin{eqnarray*}
a_{\mathfrak{q},\mathfrak{p}}(f_\mathfrak{p})(x_{1,1},\cdots,x_{1,k(1)},x_{2,1},\cdots,x_{2,k(2)},\cdots,x_{n,1},\cdots,x_{n,k(n)},y)=f_\mathfrak{p}(x_{1,1},x_{2,1},\cdots,x_{n,1},y)
\end{eqnarray*}
for each $f_\mathfrak{p}\in
L^2(\mathcal{M}_\mathfrak{p},\mu_\mathfrak{p})$ and $x_{i,j},y\in\mathcal{M}$. By $\mathcal{H}_t^T$, we denote the inductive limit of the inductive system $(\{L^2(\mathcal{M}_\mathfrak{p},\mu_\mathfrak{p})\}_{\mathfrak{p}\in\mathfrak{P}_t},\{a_{\mathfrak{q},\mathfrak{p}}\}_{\mathfrak{p}\succ\mathfrak{q}})$. Let $\kappa_{\mathfrak{p},t}:L^2(\mathcal{M}_\mathfrak{p},\mu_\mathfrak{p})\to\mathcal{H}_t^T$ be the canonical embedding. 

The family $\{U_{s,t}:\mathcal{H}_s^T\otimes^M\mathcal{H}_t^T\to\mathcal{H}_{s+t}^T\}_{s,t\geq0}$ of $M$-bilinear unitaries giving the structure of the product system $H^T=\{\mathcal{H}_t^T\}_{t\geq0}$ associated with the heat semigroup $T$ satisfies the follows: for $s,t>0,\ \mathfrak{p}\in\mathfrak{P}_s,\ \mathfrak{q}\in\mathfrak{P}_t$ with $\#\mathfrak{p}=m,\ \#\mathfrak{q}=n$ and $f_\mathfrak{p}\in
L^\infty(\mathcal{M}_\mathfrak{p},\mu_\mathfrak{p}),\ g_\mathfrak{q}\in
L^2(\mathcal{M}_\mathfrak{q},\mu_\mathfrak{q}),\ h\in
L^2(\mathcal{M})=\mathcal{H}_0^T,\ h'\in
M=L^\infty(\mathcal{M})$, we have
\begin{eqnarray*}
&&U_{s,t}((\kappa_{\mathfrak{p},s}f_\mathfrak{p})\tau^{-\frac{1}{2}}(\kappa_{\mathfrak{q},t}g_{\mathfrak{q}}))=\kappa_{\mathfrak{p}\lor\mathfrak{q},s+t}(f_\mathfrak{p}\Box
g_\mathfrak{q}),\\
&&U_{s,0}((\kappa_{\mathfrak{p},s}f_\mathfrak{p})\tau^{-\frac{1}{2}}h)=\kappa_{\mathfrak{p},s}(f_\mathfrak{p}\Box
h),\ U_{0,t}(h'\tau^{-\frac{1}{2}}(\kappa_{\mathfrak{q},t}g_{\mathfrak{q}}))=\kappa_{\mathfrak{q},t}(h'\Box
g_\mathfrak{q}).
\end{eqnarray*}
Also, the unit $\Xi^T=\{\xi^T(t)\}_{t\geq0}$ associated with $T$ is given by $\xi^T(t)=\kappa_{(t),t}1_{\mathcal{M}^2}$ for each $t>0$ and $\xi^T(0)=1_\mathcal{M}$.

For $0<s\leq
t,\ \mathfrak{p}=(s_1,\cdots,s_m)\in\mathfrak{P}_s$ and $f_\mathfrak{p}\in
L^2(\mathcal{M}_\mathfrak{p},\mu_\mathfrak{p})$, the image of $\kappa_{\mathfrak{p},s}f_\mathfrak{p}$ by the isometry $b_{t,s}:\mathcal{H}_s^T\to\mathcal{H}_t^T$ is given by $b_{t,s}(\kappa_{\mathfrak{p},s}f_\mathfrak{p})=\kappa_{(t-s)\lor\mathfrak{p},t}\tilde{f}_\mathfrak{p}$, where $\tilde{f}_\mathfrak{p}\in
L^2(\mathcal{M}_{(t-s)\lor\mathfrak{p}},\mu_{(t-s)\lor\mathfrak{p}})$ in the right side is a function defined by $\tilde{f}_\mathfrak{p}(x,x_1,\cdots,x_m,y)=f_\mathfrak{p}(x_1,\cdots,x_m,y)$ for each $x,x_i,y\in\mathcal{M}$. If we also define $\tilde{f}\in
L^2(\mathcal{M}_{(t)},\mu_{(t)})$ for $f\in
L^2(\mathcal{M})$ by $\tilde{f}(x,y)=f(y)$ for each $x,y\in\mathcal{M}$, then $b_{t,0}f=\kappa_{(t),t}\tilde{f}$. We denote the inductive limit of $(\{\mathcal{H}_t^T\}_{t\geq0},\{b_{t,s}\}_{s\leq
t})$ by $\mathfrak{H}^T$ and let $\kappa_t:\mathcal{H}_t^T\to\mathfrak{H}^T$ be the canonical embedding. We can describe the isometry $b_{t,0}^*$ by heat kernels as follows:
\begin{prop}\label{formula(t)}
For $t>0,\ \mathfrak{p}=(t_1,\cdots,t_n)\in\mathfrak{P}_t$ and $f_\mathfrak{p}\in
L^2(\mathcal{M}_\mathfrak{p},\mu_\mathfrak{p})$, we have a formula
\[
(b_{t,0}^*\kappa_{\mathfrak{p},t}f_\mathfrak{p})(y)=\int_{\mathcal{M}_{\mathfrak{p}'}}f_\mathfrak{p}(x_1,\cdots,x_n,y)p_{t_n}(x_n,y)d\mu_{\mathfrak{p}'}(x_1,\cdots,x_n)
\]
for each $y\in\mathcal{M}$, where $\mathfrak{p}'=(t_1,\cdots,t_{n-1})$.
\end{prop}
\begin{proof}
Note that the function defined by the right hand belongs to $L^2(\mathcal{M})$ by Jensen's inequality with respect to the convex function $h$ defined by $h(z)=z^2$ for each $z\in\mathbb{R}$.

For each $g\in
L^2(\mathcal{M})$, we can compute as
\begin{eqnarray*}
&&\hspace{-15pt}\langle
b_{t,0}^*\kappa_{\mathfrak{p},t}f,g\rangle=\langle\kappa_{\mathfrak{p},t}f,\kappa_{(t),t}\tilde{g}\rangle=\langle
f,a_{\mathfrak{p},(t)}\tilde{g}\rangle=\int_{\mathcal{M}_\mathfrak{p}}\overline{f_\mathfrak{p}(x_1,\cdots,x_n,y)}g(y)d\mu_\mathfrak{p}(x_1,\cdots,x_n,y)\\
&&\hspace{-15pt}=\int_{\mathcal{M}}\int_{\mathcal{M}^n}\overline{f_\mathfrak{p}(x_1,\cdots,x_n,y)}p_{t_1}(x_1,x_2)\cdots
p_{t_{n-1}}(x_{n-1},x_n)p_{t_n}(x_n,y)d\mu^n(x_1,\cdots,x_n)g(y)d\mu(y)\\
&&\hspace{-15pt}=\int_{\mathcal{M}}\overline{\int_{\mathcal{M}_{\mathfrak{p}'}}f_\mathfrak{p}(x_1,\cdots,x_n,y)p_{t_n}(x_n,y)d\mu_{\mathfrak{p}'}(x_1,\cdots,x_n)}g(y)d\mu(y).
\end{eqnarray*}
Thus, we have shown the desired equation.
\end{proof}
The right action of $M$ on the right W$^*$-$M$-module $\mathfrak{H}^T$ is given by $(\kappa_t\kappa_{\mathfrak{p},t}f_\mathfrak{p})g=\kappa_t\kappa_{\mathfrak{p},t}(f_\mathfrak{p}\Box
g)$ for each $t\geq0,\ \mathfrak{p}\in\mathfrak{P}_t,\ f\in
L^2(\mathcal{M}_\mathfrak{p},\mu_\mathfrak{p})$ and $g\in
M=L^\infty(\mathcal{M})$. Clearly, for $t>0$, the identification $\mathfrak{H}^T\otimes^M\mathcal{H}_t^T\cong\mathfrak{H}^T$ as right W$^*$-modules is obtained by the right $M$-linear unitary
\[
U_t:\mathfrak{H}^T\otimes^M\mathcal{H}_t^T\ni\kappa_s\kappa_{\mathfrak{p},s}(f_p)\tau^{-\frac{1}{2}}\kappa_{\mathfrak{q},t}(g_\mathfrak{q})\mapsto\kappa_{s+t}\kappa_{\mathfrak{p}\lor\mathfrak{q},s+t}(f_\mathfrak{p}\Box
g_\mathfrak{q})\in\mathfrak{H}^T,
\]
where $f_\mathfrak{p}\in
L^\infty(\mathcal{M}_\mathfrak{p},\mu_\mathfrak{p})$ and $g_\mathfrak{q}\in
L^2(\mathcal{M}_\mathfrak{q},\mu_\mathfrak{q})$. Also, the unitary $U_0:\mathfrak{H}^T\otimes^ML^2(\mathcal{M})\to\mathfrak{H}^T$ is the canonical isomorphism.

By the embedding (\ref{faithfulrepM}), we regard as $M\subset{\rm
End}(\mathfrak{H}_M^T)$. Then, the following direct computations imply that the triple of ${\rm
End}(\mathfrak{H}_M^T)$, the E$_0$-semigroup $\theta$ on ${\rm
End}(\mathfrak{H}_M^T)$ defined by (\ref{maxdilation}) and the projection $p=\kappa_0\kappa_0^*$ becomes a dilation of $T$. For $t>0,\ x\in\mathcal{M}$ and $f\in
M=L^\infty(\mathcal{M}),\ g\in
L^2(\mathcal{M})$, we have
\begin{eqnarray*}
&&(\kappa_0^*\theta_t(f)\kappa_0g)(x)=(\kappa_0^*U_t(f\otimes^M{\rm
id}_{\mathcal{H}_t^T})((\kappa_01_M)\tau^{-\frac{1}{2}}\kappa_{(t),t}\tilde{g}))(x)\\
&&=(\kappa_0^*U_t((\kappa_0f)\tau^{-\frac{1}{2}}\kappa_{(t),t}\tilde{g}))(x)=(\kappa_0^*\kappa_t\kappa_{(t),t}(f\Box
\tilde{g}))(x)=\int_\mathcal{M}(f\tilde{g})(y,x)p_t(y,x)d\mu(y)\\
&&=\int_\mathcal{M}f(y)g(x)p_t(y,x)d\mu(y)=(T_tf)g(x).
\end{eqnarray*}
Here, the forth equality is implied from the formula in Proposition \ref{formula(t)} in the case when $\mathfrak{p}=(t)$.

\section{Classification of E$_0$-semigroups}\label{Classification of E$_0$-semigroups}
In this section, we shall classify E$_0$-semigroups on a von Neumann algebra up to cocycle equivalence by the product systems of W$^*$-bimodules associated with their E$_0$-semigroups as CP$_0$-semigroups.

Let $\theta$ be an E$_0$-semigroup on $M$. For each $t\geq0$, let $\tilde{\mathcal{H}}^\theta_t=L^2(M)$ as sets with a left and a right actions of $M$ defined by $x\xi
y=\theta_t(x)\xi
y$ for each $x,y\in
M$ and $\xi\in\tilde{\mathcal{H}}^\theta_t$, and $\tilde{\xi}^\theta(t)=\phi^\frac{1}{2}$. Then, the family $\tilde{H}^\theta=\{\tilde{\mathcal{H}}^\theta_t\}_{t\geq0}$ is a product system of W$^*$-bimodules and $\tilde{\Xi}^{\theta}=\{\tilde{\xi}^\theta(t)\}_{t\geq0}$ is a continuous generating unital unit. 
\begin{prop}\label{canonicalsystemtheta}
Let $(H^\theta,\Xi^\theta)$ be the pair associated with $\theta$ as CP$_0$-semigroups. There is an isomorphism $u^\theta=\{u^\theta_t\}_{t\geq0}$ from $H^{\theta}$ onto $\tilde{H}^{\theta}$ preserving the units $\Xi^\theta$ and $\tilde{\Xi}^\theta$.
\end{prop}
\begin{proof}
For $t>0$, we define $u_t^\theta:\mathcal{H}_t^\theta\to\tilde{\mathcal{H}}_t^\theta$ by
\begin{eqnarray*}
&&u_t^\theta(\kappa_{\mathfrak{p},t}((x_1\otimes_{t_1}y_1\phi^\frac{1}{2})\phi^{-\frac{1}{2}}\cdots\phi^{-\frac{1}{2}}(x_n\otimes_{t_n}y_n\phi^\frac{1}{2})))\\
&&=\theta_{t_n}(\theta_{t_{n-1}}(\cdots(\theta_{t_2}(\theta_{t_1}(x_1)y_1x_2)y_2x_3)\cdots)y_{n-1}x_n)y_n
\end{eqnarray*}
for each $\mathfrak{p}=(t_1,\cdots,t_n)\in\mathfrak{P}_t$ and $x_1,\cdots,x_n,y_1,\cdots,y_n\in
M$, where $\kappa_{\mathfrak{p},t}:\mathcal{H}^\theta(\mathfrak{p},t)\to\mathcal{H}_t^\theta$ is the canonical embedding. Put $u_0^\theta={\rm
id}$. The family $u^\theta=\{u_t^\theta\}_{t\geq0}$ is the desired isomorphism. 
\end{proof}
\begin{exam}
For an E$_0$-semigroup $\theta$ on $\mathcal{B}(\mathcal{H})$, the product system $H^\theta$ associated with $\theta$ is isomorphic to the product system $\tilde{H}^\theta=\{\tilde{\mathcal{H}}_t^\theta\}_{t\geq0}$ of W$^*$-bimodules and ${\mathcal{H}}_t^\theta=\mathcal{H}\otimes\mathcal{H}^*$ with left and right actions of $\mathcal{B}(\mathcal{H})$ defined by $x(\xi\otimes\eta^*)y=(\alpha_t(x)\xi)\otimes(y^*\eta)^*$
for each $t\geq0,\ \xi,\eta\in\mathcal{H}$ and $x,y\in\mathcal{B}(\mathcal{H})$.
\end{exam}
For an E$_0$-semigroup $\theta$ on $M$, a family $\{f_t^\theta\}_{t\geq0}$ of the trivial right $M$-linear unitaries $f_t^\theta:\tilde{\mathcal{H}}_t^\theta\ni\xi\mapsto\xi\in
L^2(M)$ induces the right $M$-linear unitary $f^\theta:\tilde{\mathfrak{H}}^\theta\to
L^2(M)$, where $\tilde{\mathfrak{H}}^\theta$ is the inductive limit of $(\tilde{H}^{\theta},\tilde{\Xi}^{\theta})$. Note that the all canonical embeddings $\kappa^\theta_t:\tilde{\mathcal{H}}^\theta_t\to\tilde{\mathfrak{H}}^\theta$ are unitaries and equal to each other. We have $\theta=T^{\tilde{\Xi}^\theta}$, and the E$_0$-semigroup $\{(f^\theta)^*\theta_t(f^\theta\cdot
(f^{\theta})^*)f^\theta\}_{t\geq0}$ coincides with the dilation $\tilde{\theta}$ of the pair $(\tilde{H}^{\theta},\tilde{\Xi}^{\theta})$. We have the following classification of E$_0$-semigroups.
\begin{theo}\label{classifye0}
Two E$_0$-semigroups $\alpha$ and $\beta$ on a von Neumann algebra $M$ are cocycle equivalent if and only if $H^\alpha\cong
H^\beta$.
\end{theo}
\begin{proof}
We will use the above notations for $\alpha$ and $\beta$ in this proof. If $w=\{w_t\}_{t\geq0}\subset
M$ is a unitary right cocycle for $\alpha$ and $\beta_t(\cdot)=w_t^*\alpha_t(\cdot)w_t$ for each $t\geq0$, then $u_t:\tilde{\mathcal{H}}_t^{\beta}\ni
x\phi^\frac{1}{2}\mapsto
w_tx\phi^\frac{1}{2}\in\tilde{\mathcal{H}}_t^{\alpha}$ gives an isomorphism $\tilde{H}^{\beta}\cong\tilde{H}^{\alpha}$. Thus, we have $H^{\beta}\cong\tilde{H}^{\beta}\cong\tilde{H}^{\alpha}\cong
H^{\alpha}$.

Conversely, suppose a family $\{u_t\}_{t\geq0}$ of $M$-bilinear unitaries gives an isomorphism from $H^{\beta}$ onto $H^{\alpha}$. Let $u^\alpha$ be the isomorphism from $H^\alpha$ onto $\tilde{H}^\alpha$ in Proposition \ref{canonicalsystemtheta}. Put $\lambda(t)=u_t^\alpha
u_t\xi^{\beta}(t)\in\tilde{\mathcal{H}}_t^\alpha$ for each $t\geq0$, and then $\Lambda=\{\lambda(t)\}_{t\geq0}$ is a unital unit of $\tilde{H}^\alpha$. We have
\begin{equation}\label{theta'}
\beta_t(x)=\pi_\phi(\xi^{\beta}(t))^*\pi_\phi(x\xi^{\beta}(t))=\pi_\phi(\lambda(t))^*\pi_\phi(x\lambda(t))
\end{equation}
for all $x\in
M$, that is, $\beta=T^\Lambda$. We can check that $\Lambda$ is weakly continuous and generating. Hence the unit $\Lambda$ is continuous by Corollary \ref{strongcontinuity}. We denote the right cocycle for $\tilde{\alpha}$ associated with $\Lambda$ by $w^0=\{w_t^0=\pi_\phi(\kappa_t^\alpha\lambda(t))\pi_\phi(\kappa_0^\alpha\phi^\frac{1}{2})^*\}_{t\geq0}$ as Theorem \ref{unitcocycle}.

By (\ref{wt}), each $w_t^0$ is isometry. Since the map $:\phi^\frac{1}{2}x\mapsto\lambda(t)x$ is isometry, we have $\overline{{\rm
span}}\{\lambda(t)x\mid
x\in
M\}=L^2(M)$. Thus, by (\ref{wt0}), each $w_t^0$ is surjective. Now, we shall show that $w^0$ is strongly continuous. For $s\geq0$, by the continuity of $\Lambda$, we can check that $\kappa^\alpha_t\lambda(t)\to\kappa^\alpha_s\lambda(s)$ when $t\to
s$. Let $\xi\in\tilde{\mathfrak{H}}^\alpha$ and $t\geq
s$, by (\ref{wt}), we have
\begin{eqnarray}\label{wtwsinner}
&&\langle
w_t^0\xi,w_s^0\xi\rangle=\langle
U_{t,0}^\alpha(\lambda(t)\phi^{-\frac{1}{2}}(\kappa_0^{\alpha})^*\xi),b_{t,s}^\alpha
U_{s,0}^\alpha(\lambda(s)\phi^{-\frac{1}{2}}(\kappa_0^{\alpha})^*\xi)\rangle\nonumber\\
&&=\langle
U_{t,0}^\alpha(\lambda(t)\phi^{-\frac{1}{2}}(\kappa_0^{\alpha})^*\xi),U_{t-s,s}^\alpha({\rm
id}_{t-s}\otimes
U_{s,0}^\alpha)(\xi^\alpha(t-s)\phi^{-\frac{1}{2}}\lambda(s)\phi^{-\frac{1}{2}}(\kappa_0^{\alpha})^*\xi)\rangle\nonumber\\
&&=\langle
U_{t,0}^\alpha(\lambda(t)\phi^{-\frac{1}{2}}(\kappa_0^{\alpha})^*\xi),U_{t,0}^\alpha(U_{t-s,s}^\alpha\otimes{\rm
id}_{0})(\xi^\alpha(t-s)\phi^{-\frac{1}{2}}\lambda(s)\phi^{-\frac{1}{2}}(\kappa_0^{\alpha})^*\xi)\rangle\nonumber\\
&&=\langle\lambda(t)\phi^{-\frac{1}{2}}(\kappa_0^{\alpha})^*\xi,b_{t,s}^\alpha\lambda(s)\phi^{-\frac{1}{2}}(\kappa_0^{\alpha})^*\xi\rangle=\langle(\kappa_0^{\alpha})^*\xi,\pi_\phi(\kappa_t^\alpha\lambda(t))^*\pi_\phi(\kappa_s^\alpha\lambda(s))(\kappa_0^{\alpha})^*\xi\rangle.
\end{eqnarray}
Since $\pi_\phi(\kappa_t^\alpha\lambda(t))^*\pi_\phi(\kappa_s^\alpha\lambda(s))\to1_M$ weakly when $t\to
s$ or $s\to
t$, (\ref{wtwsinner}) tends to $\langle\xi,\xi\rangle$ when $t\to
s+0$, and by the symmetry, $\langle
w_t^0\xi,w_s^0\xi\rangle$ also tends to $\langle\xi,\xi\rangle$ when $t\to
s-0$. We conclude that $w_t^0\xi\to
w_s^0\xi$ when $t\to
s$.

Put $w_t=f^\alpha
w_t^0(f^\alpha)^*\in
M$. Then, the family $w=\{w_t\}_{t\geq0}$ is a strongly continuous right cocycle for $\alpha$. For all $t\geq0$ and $x,y,z\in
M$, since $w_t\phi^\frac{1}{2}x=\lambda(t)x$ and $\beta$ is given as (\ref{theta'}), we have $\langle
w_t^*\alpha_t(x)w_t\phi^\frac{1}{2}y,\phi^\frac{1}{2}z\rangle=\langle\pi_\phi(\lambda(t))^*\pi_\phi(x\cdot\lambda(t))\phi^\frac{1}{2}y,\phi^\frac{1}{2}z\rangle=\langle\beta_t(x)\phi^\frac{1}{2}y,\phi^\frac{1}{2}z\rangle$, and hence $\beta_t(x)=w_t^*\alpha_t(x)w_t$.
\end{proof}
\begin{coro}
Two E$_0$-semigroups $\alpha$ and $\beta$ on von Neumann algebras $M$ and $N$, respectively, are cocycle conjugate if and only if there exists a $*$-isomorphism $\Phi:M\to
N$ such that $H^\alpha$ and $H^{\beta^\Phi}$ are isomorphic.
\end{coro}
\begin{exam}\label{trivialexample}
Let $u=\{u_t\}_{t\geq0}$ be a strongly continuous semigroup of unitaries $u_t$ in a von Neumann algebra $M$. We define $\theta_t(x)=u_t^*xu_t$ for each $t\geq0$ and $x\in
M$. The product system of W$^*$-bimodules associated with $\theta$ is isomorphic to the trivial product system $\{L^2(M)\}_{t\geq0}$.
\end{exam}
\begin{prop}\label{unitcocycle}.
Let $\theta$ be an E$_0$-semigroup on a von Neumann algebra $M$. For a unit $\Xi=\{\xi(t)\}_{t\geq0}$ of $H^\theta$ and $t\geq0$, there exists a unique $a_t\in
M$ such that $\xi(t)=a_t\phi^\frac{1}{2}$. The family $\{a_t\}_{t\geq0}$ is a right cocycle for $\theta$.
\end{prop}
\begin{proof}
Fix $s,t\geq0$. Since we have
\begin{eqnarray*}
&&\phi(a_{s+t}^*(\theta_t(a_s)a_t))=\langle
a_{s+t},U_{s,t}((a_s\phi^\frac{1}{2})\phi^{-\frac{1}{2}}(a_t\phi^\frac{1}{2}))\rangle=\phi(a_{s+t}^*a_{s+t}),\\
&&\phi((\theta_t(a_s)a_t)^*a_{s+t})=\langle
U_{s,t}((a_s\phi^\frac{1}{2})\phi^{-\frac{1}{2}}(a_t\phi^\frac{1}{2})),U_{s,t}(\xi(s)\phi^{-\frac{1}{2}}\xi(t))\rangle=\phi((\theta_t(a_s)a_t)^*\theta_t(a_s)a_t)
\end{eqnarray*}
and $\phi$ is faithful, the equation $\theta_t(a_s)a_t=a_{s+t}$ holds.
\end{proof}
\begin{exam}
Let $\theta$ be an E$_0$-semigroup on a {\rm
II}$_1$ factor $M$ and $\Xi$ be a unit of the product system $H^\theta$ of W$^*$-bimodules associated with $\theta$. For each $t\geq0$, we define an operator $X_t^\Xi\in\mathcal{B}(L^2(M))$ by $X_t^\Xi(x\phi^\frac{1}{2})=\theta_t(x)a_t\phi^\frac{1}{2}$ for each $x\in
M$, where $\{a_t\}_{t\geq0}$ is the right cocycle for $\theta$ associated with $\Xi$ in {\rm
Proposition \ref{unitcocycle}}. Then, the family $X^\Xi=\{X_t^\Xi\}_{t\geq0}$ is a unit of $\theta$, that is, $X^\Xi$ is a semigroup satisfying $X_t^\Xi
x=\theta_t(x)X_t^\Xi$ for all $t\geq0$.
\end{exam}

\section*{Acknowledgments}
The author would like to express deeply gratitude to Shigeru Yamagami for his helpful comments. He also would like to thank Raman Srinivasan for a discussion. He is grateful to Yoshimichi Ueda for giving valuable informations and a chance of the discussion. He also is grateful to Yuhei Suzuki.



\begin{thebibliography}{99}
\bibitem{alev04}
A. Alevras, One parameter semigroups of endomorphisms of factors of type II$_1$, J. Operator Theory, 51 (2004), 161-179.
\bibitem{arve89}
W. Arveson, Continuous analogues of Fock space, Mem. Amer. Math. Soc. 80 (1989), no. 409, iv+66 pp.  
\bibitem{arve02}
W. Arveson, Generators of non-commutative dynamics, Ergodic Theory Dynam. Systems 22  (2002), no. 4, 1017-1030.
\bibitem{arve03}
W. Arveson, Noncommutative dynamics and E-semigroups, Springer Monographs in Mathematics, Springer-Verlag, New York (2003), x+434 pp.
\bibitem{bail-deni-have88}
M. Baillet, Y. Denizeau, J. F. Havet, Indice d'une esp\'{e}rance conditionnelle, Compositio Math. 66 (1988), no. 2, 199-236.
\bibitem{bhat96}
B. V. R. Bhat, An index theory for quantum dynamical semigroups, Trans. Amer. Math. Soc. 348 (1996), no. 2, 561-583. 
\bibitem{bhat99}
B. V. R. Bhat, Minimal dilations of quantum dynamical semigroups to semigroups of endomorphisms of C$^*$-algebras, J. Ramanujan Math. Soc. 14 (1999), no. 2, 109-124. 
\bibitem{bhat-skei00} 
B. V. R. Bhat, M. Skeide, Tensor product systems of Hilbert modules and dilations of completely positive semigroups, Infin. Dimens. Anal. Quantum Probab. Relat. Top. 3 (2000), no. 4, 519-575.
\bibitem{conn94}
A. Connes, Noncommutative Geometry. Academic Press (1994).
\bibitem{grig}
A. Grigor'yan, Heat Kernel and Analysis on Manifolds, AMS-IP Studies in Advanced Mathematics, 47, American Mathematical Society, Providence, RI; International Press, Boston, MA (2009), xvii+482 pp. 
\bibitem{lieb09}
V. Liebscher, Random sets and invariants for (type II) continuous tensor product systems of Hilbert spaces. Mem. Amer. Math. Soc. 199 (2009), no. 930, xiv+101 pp.
\bibitem{marg-srin13}
O. T. Margetts, R. Srinivasan, Invariants for E$_0$-semigroups on II$_1$ factors. Comm. Math. Phys. 323 (2013), no. 3, 1155-1184.
\bibitem{marg-srin17}
O. T. Margetts, R. Srinivasan, Non-cocycle-conjugate E$_0$-semigroups on factors, Publ. Res. Inst. Math. Sci. 53 (2017), no. 2, 299–336. 
\bibitem{muhl-sole02}
P. S. Muhly, B. Solel, Quantum Markov processes (correspondences and dilations), Internat. J. Math. 13 (2002), no. 8, 863-906. 
\bibitem{powe88}
R. T. Powers, An index theory for semigroups of $*$-endomorphisms of $\mathfrak{B}(\mathfrak{H})$, Can. Jour. Math. 40 (1988), 86-114.
\bibitem{sauv83}
J. L. Sauvageot, Sur le produit tensoriel relatif d'espaces de Hilbert, J. Operator Theory 9 (1983), no. 2, 237-252. 
\bibitem{sawa17}
Y. Sawada, A remark on the minimal dilation of the semigroup generated by a normal UCP-map, to appear in Bull. Belg. Math. Soc. Simon Stevin. 
\bibitem{sawa-yama17}
Y. Sawada, S. Yamagami, Notes on the bicategory of W$^*$-bimodules, to appear in J. Math. Soc. Japan.
\bibitem{sher00}
D. Sherman, Relative tensor products for modules over von Neumann algebras, Function spaces, 275-291, Contemp. Math. 328, Amer. Math. Soc. Providence, (2003).
\bibitem{skei03}
M. Skeide, Commutants of von Neumann modules, representations of $\mathcal{B}^a(E)$ and other topics related to product systems of Hilbert modules, Advances in quantum dynamics, Contemp. Math., 335, Amer. Math. Soc. (2003), 253-262.
\bibitem{skei06}
M. Skeide, Commutants of von Neumann correspondences and duality of Eilenberg-Watts theorems by Rieffel and by Blecher, Quantum probability, Banach Center Publ. 73, Polish Acad. Sci. Inst. Math., Warsaw, (2006), 391-408.
\bibitem{skei08}
M. Skeide, Product systems; a survey with commutants in view, Quantum stochastics and information, World Sci. Publ., Hackensack, NJ, (2008), 47-86. 
\bibitem{skei16}
M. Skeide, Classification of E$_0$-semigroups by product systems, Mem. Amer. Math. Soc. 240 (2016), no.1137, vi+126 pp.
\bibitem{take00}
M. Takesaki, Theory of operator algebras. I, Encyclopaedia of Mathematical Sciences, 124. Operator Algebras and Non-commutative Geometry, 5. Springer-Verlag, Berlin (2002), xx+415 pp.
\bibitem{take03}
M. Takesaki, Theory of operator algebras. II, Encyclopaedia of Mathematical Sciences, 125. Operator Algebras and Non-commutative Geometry, 6. Springer-Verlag, Berlin (2003), xxii+518 pp. 
\bibitem{yama92}
S. Yamagami, Algebraic aspects in modular theory, Publ. Res. Inst. Math. Sci. 28 (1992),  no. 6, 1075-1106.
\bibitem{yama94}
S. Yamagami, Modular theory for bimodules, J. Funct. Anal. 125 (1994), 327-357.
\bibitem{yama14}
S. Yamagami, Operator Algebras and Their Representations, lecture notes (2014).
\end{thebibliography}
\end{document}